\documentclass[reqno,12pt]{amsart}
\usepackage{amsmath,amssymb,epsfig,color,here,wasysym,stmaryrd,slashbox,tikz-cd,url}

\numberwithin{equation}{section}

\newtheorem{theorem}{Theorem}[section]

\newtheorem{lemma}[theorem]{Lemma}
\newtheorem{proposition}[theorem]{Proposition}
\newtheorem{definition}[theorem]{Definition}
\newtheorem{corollary}[theorem]{Corollary}

\theoremstyle{remark}
\newtheorem{example}[theorem]{Example}
\newtheorem{remark}[theorem]{Remark}

\newcommand{\rank}{\operatorname{rk}}

\newcommand{\vanish}[1]{}

\newcommand{\id}{\operatorname{id}}

\begin{document}

\title{A  Whitney polynomial for hypermaps}

\author[Robert Cori and G\'abor Hetyei]{Robert Cori \and G\'abor Hetyei}

\address{Labri, Universit\'e Bordeaux 1, 33405 Talence Cedex, France.
\hfill\break
WWW: \tt http://www.labri.fr/perso/cori/.}

\address{Department of Mathematics and Statistics,
  UNC Charlotte, Charlotte NC 28223-0001.
WWW: \tt http://webpages.charlotte.edu/ghetyei/.}

\date{\today}


\begin{abstract}
We introduce a  Whitney polynomial for hypermaps and use
it to generalize the results connecting the circuit partition
polynomial to the Martin polynomial and the results on several graph
invariants. 
\end{abstract}

\maketitle

\section*{Introduction}

The Tutte polynomial is a key invariant of graph theory, which has been
generalized to matroids, polymatroids, signed graphs and hypergraphs in
many ways. A partial list of recent topological graph and hypergraph
generalizations includes~\cite{Bernardi-embeddings, Bernardi-sandpile,
  Bernardi-Kalman-Postnikov, Kalman, Kalman-Tothmeresz}. 

The present paper generalizes a variant of the Tutte
polynomial, the Whitney rank generating function to {\em hypermaps}
which encode hypergraphs topologically embedded in a surface. This
polynomial depends on the underlying graph for maps, but different
hypermaps with the same hypergraph structure may have different Whitney
polynomials.  Our Whitney polynomial may be recursively computed using a
generalized deletion-contraction formula, and many of the famous special
substitutions (for instance, counting spanning subsets of edges, or
trees contained in the a graph) may be easily generalized to this
setting. Our approach seems to be most amenable to generalize results on
the Eulerian circuit partition polynomials, but we also have a promising
generalization of the characteristic polynomial and of the flow
polynomial. These last 
generalizations (involving the M\"obius function of the noncrossing
partition lattice) also indicate that, for hypermaps many invariants
cannot be obtained by a simple negative substitution into some
generalized Tutte polynomial. For similar reasons, a generalized Whitney
polynomial seems to work better than a generalized Tutte polynomial, and
the difference between the two should not be thought of as a mere linear
shift.   

Our paper is organized as follows. After the Preliminaries, the key
definition of our Whitney polynomial is contained in
Section~\ref{sec:whitney}. Here we also prove two recurrence formulas:
Theorem~\ref{thm:wrgrec} is a generalization of the usual recurrence
formula for the Whitney polynomial of a graph, whose recursive
evaluation yields an ever increasing number of connected components. In
Theorem~\ref{thm:wrgrecc} we show how to keep hypermaps connected: the
key ideas generalize the observation that in the computation of Tutte-Whitney
polynomials distinct components may be merged by creating a common cut
vertex. This section also contains several important specializations and
the proof of the fact that taking the dual of a planar hypermap amounts
to swapping the two variables in its Whitney polynomial.

We introduce the directed medial map of a hypermap in
Section~\ref{sec:medial}. In Proposition~\ref{prop:dmg} we show
that every directed Eulerian graph arises as the directed medial map of
a collection of hypermaps. In this section and in
Section~\ref{sec:medialc}, we generalize or prove analogues of several
results of Arratia, Bollob\'as, Ellis-Monaghan, Martin and
Sorkin~\cite{Arratia-Bollobas,Bollobas-cp,Ellis-Monaghan-mp,Ellis-Monaghan-mpmisc,Ellis-Monaghan-cpp,Ellis-Monaghan-exploring,Martin-Thesis,Martin}
on the circuit partition polynomials of Eulerian digraphs and the medial
graph of a plane graph. More precisely, we have genuine generalizations
of the results on the medial maps, and sometimes analogues on general
Eulerian digraphs. Regarding the second topic the key difference is that
a circuit partition of an Eulerian digraph represents any way to
decompose an Eulerian graph into closed paths, whereas our definitions
require that paths cannot cross at vertices. 

The visually most appealing part of our work is in
Section~\ref{sec:visual}. Here we extend the circuit partition approach
to a refined count which keeps track of circuits bounding external
(``wet'') and internal (``dry'') faces and define a process that allows
the computation of the Whitney polynomial of a planar hypermap using
paper and a scissor. 

Finally, in Sections~\ref{sec:charpoly} and \ref{sec:flowpoly} we
introduce a characteristic polynomial and a flow polynomial for
hypermaps which generalize the characteristic polynomial 
and flow polynomial of a map, as well as the characteristic polynomial
of a graded poset. We show that for hypermaps whose 
hyperedges have length at most three, these variants of the characteristic
and flow polynomials still count the admissible
colorings of the vertices, respectively the nowhere zero flows. We also
outline how to extend the duality between the chromatic polynomial and
the flow polynomial in the planar case for such hypermaps.

\section{Preliminaries}
\label{sec:prelim}

\subsection{Hypermaps}

A {\em hypermap} is a pair of permutations $(\sigma,\alpha)$
acting on the same finite set of labels, generating a transitive permutation
group. It encodes a hypergraph, topologically embedded in
a surface. Fig.~\ref{fig:hypermap} represents the planar
hypermap $(\sigma,\alpha)$ for $\sigma=(1,4)(2,5)(3)$ and
$\alpha=(1,2,3)(4,5)$. The cycles of $\sigma$ are the {\em
  vertices} of the hypermap, the cycles of $\alpha$ are its {\em
  hyperedges}. A hypermap is a {\em map} if the 
length of each cycle in $\alpha$ is at most $2$.

\begin{figure}[h]
\begin{center}
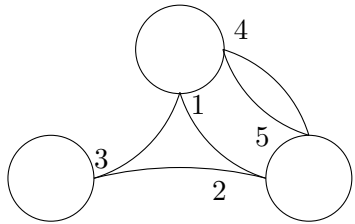
\end{center}
\caption{The hypermap $(\sigma,\alpha)$}
\label{fig:hypermap} 
\end{figure}

It is helpful to use following drawing conventions for planar
hypermaps. The cycles of $\sigma$ list the exits in
counterclockwise order, and we place the labels on the left hand side of each
exit (seen from the vertex). The points on the cycles of $\alpha$ are
listed in clockwise order. The cycles of
$\alpha^{-1}\sigma=(1,5)(2,4,3)$ label then the regions
in the plane, created by the vertices and hyperedges, these are the {\em
  faces} of the hypermap.  
There is a well-known formula, due to Jacques~\cite{Jacques} determining
the smallest genus $g(\sigma,\alpha)$ of a surface on which a hypermap
$(\sigma,\alpha)$ may be drawn. This number is given by the equation 
\begin{equation}
\label{eq:genusdef}
n + 2 -2g(\sigma,\alpha) = z(\sigma) + z(\alpha) + z(\alpha^{-1}
\sigma),
\end{equation}
where $z(\pi)$ denotes the number of cycles of the permutation $\pi$.
The number $g(\sigma,\alpha)$ is always an integer and it is called {\em
  the genus of the hypermap $(\sigma,\alpha)$}. In our example, $n=5$,
$z(\sigma)=3$, $z(\alpha)=2$, $z(\alpha^{-1}\sigma)=2$ and
Equation~\eqref{eq:genusdef} gives $g(\sigma,\alpha)=0$, that is, we
have a planar hypermap. It has been shown in~\cite[Theorem~1]{Cori} that
for a circular permutation $\sigma$ and an
arbitrary permutation $\alpha$, acting on the same set of elements, the
condition $g(\sigma,\alpha)=0$ is equivalent to requiring that the
cycles of $\alpha$ list the elements of a {\em noncrossing partition} 
according to the circular order determined by $\sigma$. We may refer to
a planar hypermap, drawn in the plane respecting the above conventions,
as a {\em plane hypermap}. This terminology is somewhat similar to the
language of graph theory where a {\em plane graph}  is a planar graph
actually drawn in the plane. It is an important difference between the
classical graph theoretic approach and our
setting that definitions (such as the {\em medial graph} of a plane graph,
see below) which in the classical literature depend on the way we
draw the graph in the plane can be extended to a planar hypermap using
the structure of labels provided. The above 
drawing rules are only a way to visualize planar hypermaps in a consistent  
fashion.  

The present paper was motivated by the study of the {\em spanning
  hypertrees of a hypermap}, initiated in~\cite{Cori-hrec, Cori-Machi,
  Cori-Penaud, Machi} and
continued in~\cite{Cori-Hetyei}. A 
hypermap $(\sigma,\alpha)$ is {\em unicellular} if it
has only one face. We call a unicellular map a {\em hypertree} if its
genus is zero. A permutation
$\beta$ is a {\em refinement} of a permutation $\alpha$, if $\beta$
obtained by replacing each cycle $\alpha_i$ of $\alpha$ by a permutation
$\beta_i$ acting on the same set of points in such a way that
$g(\alpha_i,\beta_i)=0$. We will use the notation $\beta\leq \alpha$ to
denote that the permutation $\beta$ is a refinement of the permutation
$\alpha$. A hypermap $(\sigma,\beta)$ {\em spans}
the hypermap $(\sigma,\alpha)$ if $\beta$ is a refinement of
$\alpha$. Note that not all refinements $\beta$ of $\alpha$ have the
property that $(\sigma,\beta)$ is a hypermap, because the permutation
group generated by the pair $(\sigma,\beta)$ may be not transitive. We
will use the {\em 
  hyperdeletion}  and {\em hypercontraction}  operations introduced in
~\cite{Cori-Hetyei}. A {\em hyperdeletion} is the operation of replacing
a hypermap $(\sigma,\alpha)$ with the hypermap 
$(\sigma,\alpha\delta)$ where $\delta=(i,j)$ is a transposition
{\em disconnecting $\alpha$}, that is, $i$ and $j$ must belong to the
same cycle. In this paper we will work with {\em collections of
  hypermaps}, defined in Section~\ref{sec:whitney} as pairs
$(\sigma,\alpha)$ generating a not necessarily transitive permutation
group whose orbits are the {\em connected components}. Hence we may perform
hyperdeletions even if the permutation group generated by the pair
$(\sigma,\alpha\delta)$ has more orbits than the one generated by the
pair $(\sigma,\alpha)$. For maps the deletion operation corresponds to
deleting an edge $(i,j)$, and replacing them with a pair of fixed points
of $\alpha$,    

A {\em hypercontraction} is the operation of replacing a hypermap
$(\sigma,\alpha)$ with the hypermap $(\gamma\sigma,\gamma\alpha)$ where
$\gamma=(i,j)$ is a transposition disconnecting
$\alpha$ and $i$ and $j$ belong to different cycles of
$\sigma$. For maps this operation is the contraction of an edge. 
In~\cite{Cori-Hetyei} we called these operations {\em topological
  hypercontractions} and we  also considered {\em non-topological 
  hypercontractions} which arise when $i$ and $j$ belong to the same
cycle of $\sigma$. The hypercontractions of this latter type 
correspond to a special kind of vertex
splitting for maps. We will not consider non-topological contractions in
this work.

\subsection{Circuit partition polynomials and Martin polynomials}

The circuit partition polynomials of directed and undirected Eulerian
plane graphs were defined by
Ellis-Monaghan~\cite{Ellis-Monaghan-mp,Ellis-Monaghan-mpmisc} and
independently by Arratia, Bollob\'as and Sorkin~\cite{Arratia-Bollobas} (see 
also the Introduction of~\cite{Bollobas-cp}). They are variants of  
the polynomials introduced by Martin~\cite{Martin-Thesis,Martin}. 

An undirected graph is Eulerian if each of its vertices has even
degree. A {\em closed path} in a graph is a closed walk
in which repetition of edges is not allowed, but the repetition of
vertices is. An {\em Eulerian $k$-partition} of an Eulerian graph $G$ is
a partitioning of the edge set of $G$ into $k$ closed
paths. Equivalently it is the number of {\em Eulerian graph states} with
$k$ connected components: an Eulerian graph state of an Eulerian graph
$G$ is obtained by replacing each vertex $v$ of (even) degree
$\deg(v)$ with $\deg(v)/2$ vertices
of degree $2$, each new vertex is incident to a pair of edges originally
incident to $v$ in such a way that the resulting pairs of edges
partition the set of edges originally incident to $v$.
The following definition of the Martin polynomial of an undirected
Eulerian graph is given in~\cite{LasVergnas-mp} (see
also~\cite{Ellis-Monaghan-mp}):  
\begin{equation}
M(G;x)=\sum_{k,m\geq 0} f_{(k+1,m)}(G) \cdot 2^{b(G)-m} (x-2)^k 
\end{equation}  
where $b(G)$ is the number of loops in $G$, and $f_{(k+1,m)}(G)$ is the
number of Eulerian $k$-partitions of $G$, containing $m$ closed paths of
length $1$. The {\em circuit partition polynomial $J(G,x)$} of an undirected
Eulerian graph is given in~\cite{Ellis-Monaghan-cpp} by
\begin{equation}
  J(G;x)=\sum_{k\geq 0} f_k(G) x^k 
\end{equation}  
where $f_k(G)$ is the number of Eulerian states with $k$ components. 

A directed graph is Eulerian if the in-degree of each vertex is
equal to its out-degree. An Eulerian graph state of an Eulerian directed graph
$\overrightarrow{G}$ is defined similarly to the undirected graph state,
except each set of directed edges incident to a vertex $v$ must be
partitioned in a way that an incoming edge is always paired with an
outgoing edge. Hence an Eulerian graph state is always a disjoint union
of directed cycles.   The Martin polynomial of a directed Eulerian digraph
$\overrightarrow{G}$ is given in~\cite{LasVergnas-mp} by  
\begin{equation}
m(\overrightarrow{G};x)=\sum_{k\geq 0} f_{k+1}(\overrightarrow{G}) 
(x-1)^k
\end{equation}  
where $f_{k}(\overrightarrow{G})$ is the number of ways to partition the
edge set of $\overrightarrow{G}$ into $k$ disjoint cycles. In the
terminology of~\cite{Ellis-Monaghan-cpp}, it the number of Eulerian graph states
of $\overrightarrow{G}$ with $k$ components. The circuit partition
polynomial of an Eulerian digraph is defined
in~\cite{Ellis-Monaghan-cpp} by  
\begin{equation}
  j(\overrightarrow{G},x)=\sum_{k\geq 0} f_k(\overrightarrow{G}) x^k.
\end{equation}
It is an immediate consequence of the definitions that in the directed
setting we have
\begin{equation}
\label{eq:dccp}  
  j(\overrightarrow{G},x)=xm(\overrightarrow{G},x+1),
\end{equation}  
see~\cite[Eq. (1)]{Ellis-Monaghan-cpp}. 

An important example of an undirected Eulerian graph is the {\em medial graph}
$G_m$ of a plane graph $G$ is obtained by  putting a vertex on each edge
of the graph and drawing edges around the faces of $G$. The resulting
medial graph is always a $4$-regular graph, and it is connected if $G$
is connected. The {\em directed medial graph $\overrightarrow{G_m}$} is
obtained by directing the edges of $G_m$ in such a way that each
shortest circuit enclosing an original vertex of $G$ is directed
counterclockwise. The graph $\overrightarrow{G_m}$ is an Eulerian digraph.
It has been shown by Martin~\cite{Martin-Thesis,Martin} that the Tutte
polynomial $T(G;x,y)$ of a connected plane graph $G$ is connected to the Martin
polynomial of its directed medial graph $\overrightarrow{G_m}$ by the
formula
\begin{equation}
\label{eq:MartinT}  
m(\overrightarrow{G_{m}};x)=T(G;x,x).
\end{equation}  
The following generalization of this formula to not necessarily connected plane
graphs may be found in~\cite[Eq.\ (15)]{Ellis-Monaghan-cpp}: 
\begin{equation}
\label{eq:cppT}    
j(\overrightarrow{G_{m}};x)=x^{c(G)} T(G;x+1,x+1)
\end{equation}  
where $c(G)$ is the number of connected components of $G$.

\section{A Whitney polynomial of a collection of hypermaps}
\label{sec:whitney}

\begin{definition}
A {\em collection of hypermaps $(\sigma,\alpha)$} is an ordered pair of
permutations acting on the same set of points. We call the
orbits of the permutation group generated by $\sigma$ and $\alpha$ the
{\em connected components of $(\sigma,\alpha)$} and denote their number 
by $\kappa(\sigma,\alpha)$. 
\end{definition}  
Clearly the collection of hypermaps is a hypermap if and only if
$\kappa(\sigma,\alpha)=1$, that is, the permutations $\sigma$ and
$\alpha$ generate a transitive permutation group.  
\begin{definition}
\label{def:wrg}  
The {\em Whitney polynomial  $R(\sigma,\alpha;u,v)$} of a
collection of hypermaps $(\sigma,\alpha)$ on a set of $n$ points is
defined by the formula 
$$
R(\sigma,\alpha;u,v)=\sum_{\beta\leq\alpha}
u^{\kappa(\sigma,\beta)-\kappa(\sigma,\alpha)}\cdot
v^{\kappa(\sigma,\beta)+n-z(\beta)-z(\sigma)}
$$
Here the summation is over all permutations $\beta$ refining $\alpha$.
\end{definition}
Note that, for maps, Definition~\ref{def:wrg} coincides with usual
definition of the Whitney polynomial of the underlying
graph $G$. Indeed, the two-cycles of the permutation $\beta$ refining
$\alpha$ may be identified with a subset of the edge set of $G$ and the
number of edges in this subset is $n-z(\beta)$.
\begin{remark}
The Whitney polynomial of a graph is often called the {\em Whitney rank
  generating function}. The rank is the rank of the incidence
matrix of the graph over a field of characteristic zero, it is the
difference between the number of vertices and the number of connected
components. In the present setting this rank corresponds to
taking $z(\sigma)-\kappa(\sigma,\alpha)$ as the rank of a collection of
hypermaps: if we consider the cycles of $\sigma$ as vertices, and we
introduce for each ordered pair $(i,\alpha(i))$ a directed edge from the vertex
containing $i$ to the vertex containing $\alpha(i)$, the rank of the
incidence matrix of the resulting directed graph will be exactly
$z(\sigma)-\kappa(\sigma,\alpha)$. 
\end{remark}

Hypermaps are topological embeddings of generalized hypergraphs, and one
would be tempted to think that different hypermaps with the same
underlying hypergraph structure would have the same Whitney
polynomial. This is not the case, as it is shown in the next example. 
\begin{example}
Consider the hypermaps $(\sigma, \alpha)$ and $(\sigma, \alpha')$ shown
in Figure~\ref{fig:hypermap-pair}. These are given by 
$\sigma=(1,5)(2,6)(3)(4)$, $\alpha=(1,2,3,4)(5,6)$ and
$\alpha'=(1,4,2,3)(5,6)$. Both hypermaps have the same $4$ vertices,
each of them has the same edge $(5,6)$ connecting the same two vertices
and a hyperedge connecting all $4$ vertices. Despite being topological
representations of the same hypergraph, the constant term of
$R(\sigma,\alpha;u,v)$ is $5$ (contributed by the refinements
$(1,2,3,4)(5)(6)$, $(1,2,3,4)(5,6)$, 
$(1,3,4)(2)(5,6)$, $(1)(2,3,4)(5,6)$ and $(1,4)(2,3)(5,6)$ of $\alpha$),
whereas the constant term of $R(\sigma,\alpha';u,v)$ is $6$ (contributed by
the refinements $(1,4,2,3)(5)(6)$, $(1,4,2,3)(5,6)$,  
$(1,4,3)(2)(5,6)$, $(1)(2,3,4)(5,6)$, $(1,4)(2,3)(5,6)$ and
$(1,3)(2,4)(5,6)$). Indeed, the exponent of $u$ is zero if and only if
the refinement $\beta$ satisfies $\kappa(\sigma,\beta)=1$, that is
$(\sigma,\beta)$ must be a hypermap. Given this property, the exponent of
$v$ is zero exactly when $\beta$ has $3$ cycles.      
\end{example}  
\begin{figure}[h]
\begin{center}
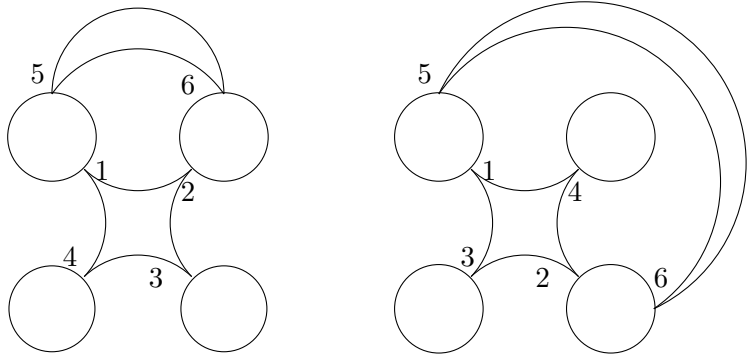
\end{center}
\caption{A pair of hypermaps whose Whitney polynomials are different}
\label{fig:hypermap-pair} 
\end{figure}

The multiplicative property of the Tutte polynomial of maps has the
following obvious generalization. 
\begin{proposition}
\label{prop:product}  
Suppose the collections of hypermaps $(\sigma_1,\alpha_1)$  and
$(\sigma_2,\alpha_2)$ act on disjoint sets of points $P_1=\{1,2,\ldots,n_1\}$
respectively $P_2=\{1',2',\ldots,n_2'\}$. Let $(\sigma,\alpha)$ be the
collection of hypermaps acting on $P_1\cup P_2$ in such a way that the
restriction of $\sigma$ to $P_i$ is $\sigma_i$ and the restriction of
$\alpha$ to $P_i$ is $\alpha_i$ for $i=1,2$. Then
$$
R(\sigma,\alpha;u,v)=R(\sigma_1,\alpha_1;u,v)\cdot R(\sigma_2,\alpha_2;u,v)
$$
\end{proposition}
\begin{proof}
The restrictions $\beta_i$ to $P_i$  of a
refinement $\beta$ of $\alpha$ may be chosen independently for $i=1,2$.
Furthermore, we have
$\kappa(\sigma,\beta)=\kappa(\sigma_1,\beta_1)+\kappa(\sigma_2,\beta_2)$,
$\kappa(\sigma,\alpha)=\kappa(\sigma_1,\alpha_1)+\kappa(\sigma_2,\alpha_2)$,
$z(\alpha)=z(\alpha_1)+z(\alpha_2)$, $z(\beta)=z(\beta_1)+z(\beta_2)$, 
$z(\sigma)=z(\sigma_1)+z(\sigma_2)$, and $n=n_1+n_2$. 
\end{proof}
The function
$R(\sigma,\alpha;u,v)$ may be computed recursively using the generalized
deletion-contraction formula stated in Theorem~\ref{thm:wrgrec}  below.
In the proof of this theorem the following lemma plays a key role.
\begin{lemma}
\label{lem:refinement}
Consider the circular permutation $\alpha=(1,2,\ldots,m)$ for some
$m\geq 2$ and let $\beta$ be a refinement of $\alpha$ satisfying
$\beta(1)=k$. Then the permutation
\begin{equation}
\label{eq:beta'}
\beta'=\begin{cases}
\beta & \mbox{if $k=1$}\\
(1,k)\beta & \mbox{otherwise}\\
\end{cases}
\end{equation}
is a refinement of $(1)(2,3,\ldots,m)$. Furthermore, for $k\not\in \{1,2\}$, the
permutation $\beta'$ is also a refinement of
$(1)(2,\ldots,k-1)(k,k+1,\ldots,m)$. 
\end{lemma}  
\begin{proof}
Since $\beta$ takes $1$ into $k$, the
point $1$ is always a fixpoint of $\beta'$ and $\beta'$ is always a
refinement of $(1)(2,\ldots,m)$. In the case when  
$k\not\in\{1,2\}$ holds, the cycles of $\beta$ refine the permutation
$(1,k,k+1,\ldots,m)(2,\ldots,k-1)$, and the cycles of
$\beta'=(1,k)\beta$ refine the permutation
$(1)(k,k+1,\ldots,m)(2,\ldots,k-1)$. 
\end{proof}  

To state our generalized deletion-contraction formula we now define
operations on a collection of hypermaps.  
\begin{definition}
\label{def:phik}
Let $(\sigma,\alpha)$ be a collection of hypermaps and assume that
$(1,2,\ldots,m)$ is a cycle of $\alpha$ of length at least $2$. For each
$k\in\{1,2,\ldots,m\}$ we define the collection of hypermaps
$\phi_k(H)=(\sigma_k,\alpha_k)$ where
$$
\sigma_k=
\begin{cases}
  \sigma & \mbox{if $1$ and $k$ belong to the same cycle of $\sigma$,}\\
  (1,k)\sigma & \mbox{otherwise;}\\
\end{cases}  
$$
and the permutation $\alpha_k$ is obtained from $\alpha$ by replacing
  the cycle $(1,2,\ldots,m)$  with 
  $(1)(2,\ldots,m)$ if $k\in\{1,2\}$ and 
  $(1)(2,\ldots,k-1)(k,\ldots,m)$ if $k\not\in\{1,2\}$.
\end{definition}
It is worth noting that the definition of $\phi_k(H)$ may be
equivalently rewritten as follows. 
\begin{equation}
\label{eq:Hrule}  
\phi_k(H)=
\begin{cases}
((1,k)\sigma,(1,k)\alpha(1,k-1))& \mbox{if $z((1,k)\sigma)\leq z(\sigma)$},\\
(\sigma,(1,k)\alpha(1,k-1))& \mbox{otherwise}.  
\end{cases}  
\end{equation}
In rule~(\ref{eq:Hrule}) we count modulo $m$, that is, we replace $k-1$
  with $m$ if $k=1$, and we read $(1,1)$ as a shorthand for the identity
  permutation. In particular, for a collection of maps we must have $m=2$, the
  collection $\phi_1(\sigma,\alpha)$ is obtained from $(\sigma,\alpha)$
  by deleting the edge $(1,2)$. Furthermore, in the case when $(1,2)$ is
  not a loop, $\phi_2(\sigma,\alpha)$ is obtained from $(\sigma,\alpha)$
  is obtained from $(\sigma,\alpha)$ by contracting $(1,2)$.

\begin{theorem}
\label{thm:wrgrec}  
Let $H=(\sigma,\alpha)$ be a collection of hypermaps on the set
$\{1,2,\ldots,n\}$ and assume that $(1,2,\ldots,m)$ is a cycle of
$\alpha$ of length at least $2$. Then the Whitney polynomial $R(H;u,v)$
is given by the sum
$$
R(H;u,v)=\sum_{k=1}^m R(\phi_k(H);u,v)\cdot w_k,
$$
where each $w_k$ is a monomial from the set $\{1,u,v,uv\}$, given by the
equations 
\begin{equation}
\label{eq:wrule}
w_k=
\begin{cases}
u^{\kappa(\phi_k(H))-\kappa(H)}v & \mbox{if $k\neq 1$ and $1$ and $k$ belong to
  the same cycle of $\sigma$};\\  
u^{\kappa(\phi_k(H))-\kappa(H)}& \mbox{otherwise}.\\
\end{cases}  
\end{equation}
\end{theorem}  
\begin{proof}
For each $k\in\{1,2,\ldots,m\}$, we use Lemma~\ref{lem:refinement} to
establish a bijection between all refinements $\beta$ of $\alpha$
satisfying $\beta(1)=k$ and all refinements $\beta'$ of $\alpha_k$ in
such a way that the contribution $w_{\beta}$ of $\beta$ to $R(H;u,v)$
is the same as the contribution $w_{\beta'}$ of $\beta'$ to
$R(\phi_k(H);u,v)$ multiplied by $w_k$.  
We define $\beta'$ using Equation~(\ref{eq:beta'})
and apply Lemma~\ref{lem:refinement} to the cycles of $\beta$ and
$\beta'$ contained in the set $\{1,2,\ldots,m\}$. 
As seen in  Lemma~\ref{lem:refinement}, $\beta'$ has  $1$ as a fixed
point hence it is a refinement of the permutation $\alpha_k$.
Conversely, given any permutation $\beta'$ that
is a refinement of $\alpha_k$, the permutation
$\beta=(1,k)\beta'$ is a refinement of $\alpha$ sending $1$ into
$k$. The involution $\beta\mapsto \beta'$ is a bijection between all
refinements $\beta$ of $\alpha$ taking $1$ into $k$ and all refinements
$\beta'$ of  $\alpha_k$. We only need to verify that formula~\eqref{eq:wrule} is
correct and it always yields a monomial from the set $\{1,u,v,uv\}$.
We distinguish three cases. 

{\bf\noindent Case 1:} $k=1$ holds. In this case we have $\sigma_1=\sigma$,
$\beta_1=\beta$, and $\alpha_1$ is obtained from $\alpha$ by replacing
the cycle $(1,2,\ldots,m)$ with $(1)(2,\ldots,m)$. We have
$\kappa(\sigma_1,\beta_1)=\kappa(\sigma,\beta)$ hence the exponent of $u$
in $w_1$ must be $\kappa(\phi_1(H))-\kappa(H)$. The exponent of $v$
is the same, term by term, in $w_\beta$ and $w_{\beta'}$, hence the
exponent of $v$ in $w_1$ is zero. Note finally that $\phi_1(H)$ is
obtained from $H$ by replacing the cycle $(1,2,\ldots,m)$ of $\alpha$
with two cycles, this increases the number of connected components by at
most $1$. We obtain that $w_1\in\{1,u\}$.   

{\bf\noindent Case 2:} $k>1$ holds and $1$ and $k$ belong to
different cycles of $\sigma$. In this case the orbits
of the permutation group generated by $\sigma$ and $\alpha$ are the same
as those of the permutation group generated by $\sigma_k=(1,k)\sigma$
and $(1,k)\alpha$: $(1,k)\alpha$ is obtained from
$\alpha$ by replacing $(1,\ldots,m)$ with $(1,\ldots,k-1)(k,\ldots,m)$,
the cycle containing $1$ and $k$ is broken into two cycles, but
replacing $\sigma$ with $\sigma_k$ merges the two cycles of $\sigma$
that contain $1$ and $k$. Similarly, the orbits of the permutation group
generated by $\sigma$ and $\beta$ are the same as those of the permutation group
generated by $\sigma_k$ and $\beta'=(1,k)\beta$. Since we also have
$z(\beta')=z(\beta)+1$ and $z(\sigma_k)=z(\sigma)-1$, we obtain
$$w_{\beta}=
u^{\kappa(\sigma,\beta)-\kappa(\sigma,\alpha)}\cdot
v^{\kappa(\sigma,\beta)+n-z(\beta)-z(\sigma)}
=u^{\kappa(\sigma_k,\beta')-\kappa(\sigma_k,(1,k)\alpha)}\cdot
v^{\kappa(\sigma_k,\beta')+n-z(\beta')-z(\sigma_k)}.$$ 
Replacing $u^{-\kappa(\sigma_k,(1,k)\alpha)}$ with
$u^{-\kappa(\sigma_k,\alpha_k)}$  yields an additional factor of $u$
if and only if breaking the cycle $(1,\ldots,k-1)$ of $(1,k)\alpha$ into
$(1)(2,\ldots,k-1)$ increases the number of connected components by
one. We obtain $w_k\in \{1,u\}$. 

{\bf\noindent Case 3:} $k\neq 1$ and the points $1$ and $k$ belong to
the same cycle of $\sigma$. In this case the orbits of the permutation
group generated by $\sigma$ and $\alpha$ are the same as those of the
permutation group generated by $\sigma_k=\sigma$ and $(1,k)\alpha$. Indeed,
the cycles of $(1,k)\alpha$ are obtained by replacing the cycle
$(1,2,\ldots,m)$ with the cycles $(1,2,\ldots,k-1)$ and
$(k,\ldots,m)$. Hence each orbit of the permutation group generated by
$\sigma$ and $(1,k)\alpha$ is contained in an orbit of the
permutation group generated by $\sigma$ and $\alpha$. We only need to
show that subdividing the cycle $(1,2,\ldots,m)$ into two cycles does
not increase the number of orbits. This is obviously true, as we may use
the cycle of $\sigma$ that contains both $1$ and $k$ to reach $k$ from
$1$. Similarly, the orbits of the permutation
group generated by $\sigma$ and $\beta$ are the same as those of the
permutation group generated by $\sigma_k$ and $\beta'=(1,k)\beta$. So far we
obtained that $\kappa(\sigma,\alpha)=\kappa(\sigma_k,(1,k)\alpha)$ and
$\kappa(\sigma,\beta)=\kappa(\sigma_k,\beta')$ hold. Observe finally
that $z(\beta')=z(\beta)+1$ since $\beta$ contains a cycle in which
$1$ is immediately followed by $k$ and $\beta'=(1,k)\beta$ is obtained by
breaking off the cycle $(1)$ from this cycle. Hence to obtain the
contribution $w_{\beta}$ of $(\sigma,\beta)$ to $R(H;u,v)$ from the
contribution $w(\beta')$ of $(\sigma_k,\beta')$ to
$R(\phi_k(H);u,v)$ we must always add a factor of $v$ in
this case. Similarly to the previous case, we must also add a factor of
$u$ whenever
$\kappa(\sigma,(1,k)\alpha)+1=\kappa(\sigma,\alpha_k)$
holds. We obtain $w_k\in \{v,uv\}$.
\end{proof}  

\begin{example}
Consider the case when $(\sigma,\alpha)$ is a map and $(1,2)$ is a cycle
of $\alpha$. The substitution $k=1$ is covered in Case~1 of the proof of
Theorem~\ref{thm:wrgrec} above. The collection of hypermaps
$(\sigma,\alpha(1,2))$ is obtained from $(\sigma,\alpha)$ by deleting
the edge corresponding to the cycle $(1,2)$ and we must add a factor of
$u$ exactly when this edge is a bridge. The case covering
substitution $k=2$ depends on the relative position of $1$ and $2$ on
the cycles of $\sigma$. If the edge corresponding to the cycle $(1,2)$
is a loop then we are in Case~3 and the collection of hypermaps
$(\sigma,(1,2)\alpha)$ is obtained from $(\sigma,\alpha)$ by deleting
this loop edge in the underlying graph. If the edge corresponding to
$(1,2)$ is not a loop then the collection of hypermaps
$((1,2)\sigma,(1,2)\alpha)$ is obtained from $(\sigma,\alpha)$ by
contracting this edge (no additional factor of $u$ will occur as
$(1,2)\alpha(1,1)=(1,2)\alpha$). We obtain the following well-known
recurrences for the Whitney polynomial $R(G;u,v)$ of the
underlying graph $G$:
$$R(G;u,v)=
\begin{cases}
  (1+u)R(G/e;u,v)&\text{if $e$ is a bridge;}\\
  (1+v)R(G-e;u,v)&\text{if $e$ is a loop;}\\
  R(G-e;u,v)+R(G/e;u,v)&\text{otherwise}.\\
\end{cases}
$$
\end{example}
In analogy to the case of maps, Theorem~\ref{thm:wrgrec} may be modified
in such a way that for a hypermap $(\sigma,\alpha)$ the recurrence only
involves hypermaps. The validity of the modified recurrence depends on
the following observation, generalizing the fact that the Tutte polynomial
of a graph does not change if we merge two vertices that belong to
different connected components.  
\begin{lemma}
\label{lem:merge}  
Let $(\sigma,\alpha)$ be a collection of hypermaps with at least two
connected components. If the points $i$ and $j$ belong to different
connected components then
$$
R(\sigma,\alpha;u,v)=R((i,j)\sigma,\alpha;u,v).
$$
\end{lemma}
\begin{proof}
To compute either of the two polynomials, we must sum over the same
refinements $\beta$ of the same permutation $\alpha$ on both sides. The
numbers $n$ and $z(\beta)$ are the same on both sides, and we have
$\kappa(\sigma,\alpha)=\kappa((i,j)\sigma,\alpha)+1$,
$\kappa(\sigma,\beta)=\kappa((i,j)\sigma,\beta)+1$, and
$z(\sigma)=z((i,j)\sigma)+1$.   
\end{proof}
Using~\eqref{eq:Hrule} and Lemma~\ref{lem:merge} we may restate
Theorem~\ref{thm:wrgrec} as follows.
\begin{theorem}
\label{thm:wrgrecc}  
Let $H=(\sigma,\alpha)$ be a collection of hypermaps on the set
$\{1,2,\ldots,n\}$ and assume that $(1,2,\ldots,m)$ is a cycle of
$\alpha$. Then the Whitney polynomial $R(H;u,v)$ is given by the sum
$$
R(H;u,v)=\sum_{k=1}^m R(\psi_k(H);u,v)\cdot w_k,
$$
where each monomial $w_k$ is given by the rule~(\ref{eq:wrule}) and each
$\psi_k(H)$ is a collection of hypermaps defined as follows:
$$
\psi_k(H) =
\begin{cases}
((1,k)\sigma,(1,k)\alpha(1,k-1))& \mbox{if $z((1,k)\sigma)\leq z(\sigma)$
    and $\kappa(\phi_k(H))=\kappa(H)$},\\
((1,2)(1,k)\sigma,(1,2)(1,k)\alpha)& \mbox{if $z((1,k)\sigma)\leq z(\sigma)$
    and $\kappa(\phi_k(H))=\kappa(H)+1$},\\
(\sigma,\alpha(1,k-1)(1,m)) & \mbox{if $z((1,k)\sigma)=z(\sigma)+1$
    and $\kappa(\phi_k(H))=\kappa(H)$,}\\ 
((1,2)\sigma,(1,2)\alpha(k-1,m)) &\mbox{if $z((1,k)\sigma)=z(\sigma)+1$
    and $\kappa(\phi_k(H))=\kappa(H)+1$.}\\
\end{cases}  
$$
Here we count modulo $m$.
\end{theorem}
\begin{proof}
Let us compare the present formula with the one given in
Theorem~\ref{thm:wrgrec}, using the equivalent definition of $\phi_k(H)$
stated in~\eqref{eq:Hrule}. The definition of $\psi_k(H)$ is the same as
that of $\phi_k(H)$ on the first and the third lines, we only rewrote
$(1,k)\alpha(1,k-1)$ as $\alpha(1,k-1)(1,m)$ on the third line. The two
permutations are equal: 
$$
(1,k)(1,2\ldots,m)(1,k-1)=(1,2\ldots,m)(1,k-1)(1,m)$$
as both products give $(1)(2,\ldots,k-1)(k,\ldots,m)$. 
The reason of this rewriting will be explained in
Remark~\ref{rem:wrgrecc} below. Similarly, the permutation 
$(1,k)\alpha(1,k-1)$ also equals $(1,2)(1,k)\alpha$ for $k\neq 2$ and
with $(1,2)\alpha(k-1,m)$ for $k\not\in \{1,2\}$ on the respective
second and fourth lines. The only difference between $\psi_k(H)$ and
$\phi_k(H)$ on the second and fourth lines is that we replaced each
$\sigma$ with $(1,2)\sigma$ in Theorem~\ref{thm:wrgrecc}. These are the
cases when after the deletion of $(1,k-1)$ the points $1$
and $2$ belong to different connected components. Hence we may apply
Lemma~\ref{lem:merge}. 
\end{proof}  
\begin{remark}
\label{rem:wrgrecc}
Theorem~\ref{thm:wrgrecc} may be restated as follows. Let
$H=(\sigma,\alpha)$ by a collection of hypermaps, and  
$(1,2,\ldots,m)$ a cycle of $\alpha$. The polynomial
$R(H;u,v)$ may be recursively computed as the contribution
of $m$ collections of hypermaps $\psi_k(H)$, each corresponding to a
$k\in\{1,2,\ldots,m\}$.
\begin{enumerate}
\item If $1$ and $k$ belong to different vertices (or $k=1$) and deleting
  $(1,k-1)$ doesn't increase the number of connected components, then the
  $\psi_k(H)$ is obtained
  by contracting $(1,k)$ (or doing nothing when $k=1$) and deleting
  $(1,k-1)$, and it contributes $R(\psi_k(H);u,v)$.
\item If $1$ and $k$ belong to different vertices (or $k=1$) but deleting
  $(1,k-1)$ increases the number of connected components, then
  $\psi_k(H)$ is obtained by contracting first $(1,k)$ and then $(1,2)$,
  and it contributes $R(\psi_k(H);u,v)\cdot u$. Note
  that we can also contract $(1,2)$ and then $(2,k)$, as $(1,2)(1,k)$
  and $(2,k)(1,2)$ both equal $(1,k,2)$.
\item If $1$ and $k\geq 2$ belong to the same vertex but deleting
  $(1,k-1)$ doesn't increase the number of connected components, then
    $\psi_k(H)$ is obtained by first deleting
  $(1,k-1)$ and then $(1,m)$, and it contributes $R(\psi_k(H);u,v)\cdot
  v$. Note that we could also first delete $(1,m)$ and then $(k-1,m)$ as
  $(1,k-1)(1,m)$ and $(1,m)(k-1,m)$ both equal $(1,m,k-1)$.   
\item If $1$ and $k\geq 3$ belong to the same vertex and deleting
  $(1,k-1)$ increases the number of connected components, then the
  $\psi_k(H)$ is  obtained by contracting $(1,2)$ and deleting
  $(k-1,m)$, and it contributes $R(\psi_k(H);u,v)\cdot uv$. 
\end{enumerate}  
It is easy to verify that in the case when $H=(\sigma,\alpha)$ is a single
hypermap the same holds for $\psi_k(H)$ in each case above. Indeed, the
number of the connected components can only increase when we delete
$(1,k-1)$. Deleting $(1,k-1)$ replaces the 
cycle $(1,2,\ldots,m)$ of $\alpha$ with $(2,\ldots,k-1)(1,k,\ldots,m)$.
These are precisely the cases when we replace $\sigma$ (respectively
$(1,k)\sigma$) with $(1,2)\sigma$ (respectively
$(1,2)(1,k)\sigma$) using Lemma~\ref{lem:merge}. Hence the points $1$
and $2$ remain connected through the vertices. 
\end{remark}  
\begin{example}
Consider the hypermap $((1,4)(2,5)(3),(1,2,3)(4,5))$. Applying
Theorem~\ref{thm:wrgrec} to the cycle $(1,2,3)$ we obtain 
\begin{align*}
R((1,4)(2,5)(3),(1,2,3)(4,5);u,v)&=
R((1,4)(2,5)(3),(1)(2,3)(4,5);u,v)\\&+
R((1,4,2,5)(3),(1)(2,3)(4,5);u,v)\\
&+R((1,4,3)(2,5),(1)(2)(3)(4,5);u,v).
\end{align*}
On the right hand side we obtain Whitney polynomials of
maps, they are $(u+1)^2$, $(u+1)(v+1)$ and $u+1$ respectively.
Hence
$$
R((1,4)(2,5)(3),(1,2,3)(4,5);u,v)=u^2+uv+4u+v+3.
$$
\end{example}  
\begin{example}
A simplest example of a calculation where $w_k=uv$ occurs in 
Theorems~\ref{thm:wrgrec} and~\ref{thm:wrgrecc} is when
we compute the Whitney polynomial of the
hypermap $((1,3)(2), (1,2,3))$. For $k=3$ the points $1$ and $3$ belong
to the same vertex $(1,3)$ whereas deleting $(1,2)$ yields the
collection of hypermaps $((1,3)(2),(1,3)(2))$ which has two connected
components. Hence  $((1,3)(2),(1,3)(2))$ contributes
$$R((1,3)(2),(1,3)(1,2,3)(1,2); u,v)\cdot uv
=R((1,3)(2),(1)(2)(3); u,v)\cdot uv=uv$$
to $R((1,3)(2), (1,2,3);u,v)$.
\end{example}  

It is well known that certain substitutions into the Tutte polynomial
yield famous graph theoretic invariants, such as the number of spanning
forests, of the spanning subgraphs, and so on. Some of these results carry over
easily to the Whitney polynomial of a collection of hypermaps. 

The exponent of $v$ in Definition~\ref{def:wrg} may be rewritten as
\begin{equation}
\label{eq:expofv}  
\kappa(\sigma,\beta)+n-z(\beta)-z(\sigma)=2g(\sigma,\beta)+z(\beta^{-1}\sigma)-\kappa(\sigma,\beta)
\end{equation}  
where $g(\sigma,\beta)$ is the sum of the genuses of the hypermaps
constituting $(\sigma,\beta)$. Since each orbit of the permutation group
generated by $\sigma$ and $\beta$ contains at least one cycle of
$\beta^{-1}\sigma$, this exponent is $0$ if and only if $(\sigma,\beta)$
is a {\em hyperforest}: a collection of genus zero unicellular
hypermaps. 

The exponent $\kappa(\sigma,\beta)-\kappa(\sigma,\alpha)$ of $u$
in Definition~\ref{def:wrg} is zero if and only if the subgroup
generated by $\sigma$ and $\beta$ has the same orbits as the permutation
group generated by $\sigma$ and $\alpha$. We call such a collection of
hypermaps $(\sigma,\beta)$ a {\em spanning collection of hypermaps}. 
\begin{proposition}
\label{prop:sphyp}  
The following specializations hold for the Whitney polynomial of any
collection of hypermaps $(\sigma,\alpha)$:
\begin{enumerate}
\item $R(\sigma,\alpha;0,0)$ is the number of spanning hyperforests of
  $(\sigma,\alpha)$. 
\item $R(\sigma,\alpha;0,1)$ is the number of spanning collections of
  hypermaps of $(\sigma,\alpha)$.  
\end{enumerate}
\end{proposition}
The straightforward verification is left to the reader. Note that in the
special case when $\kappa(\sigma,\alpha)=1$, that is, in the case when 
$(\sigma,\alpha)$ is a hypermap, Proposition~\ref{prop:sphyp} implies
that $R(\sigma,\alpha;0,0)$ is the number of spanning {\em hypertrees}
as defined in~\cite{Cori-Hetyei}.   

It is also well
known that the Tutte polynomial $T(G;x,y)$ of a graph is 
closely related to its Whitney polynomial $R(G;u,v)$ by
the formula $T(G;x,y)=R(G;x-1,y-1)$. The class of all hypermaps contains
all maps, hence the result~\cite[(9.1) Theorem]{Welsh} on the
\#P-hardness of most evaluations of the Tutte polynomial implies the following.
\begin{corollary}
\label{cor:sharpP}  
The evaluation of $R(\sigma,\alpha;u,v)$ at $u=a$ and $v=b$ is \#P-hard
if $(a,b)$ does not belong to the special hyperbola $uv=1$, nor to the set
$$\{(0,0),(-2,-2), (-1,-2),(-2,-1), (i-1,-i-1), (\rho-1,\rho^2-1),
(\rho^2-1,\rho-1)\}$$
where $\rho$ is a primitive complex third root of unity. 
\end{corollary}  
For the Tutte polynomial of graphs it is also known that the exceptional
substitutions listed in Corollary~\ref{cor:sharpP} may be computed in
polynomial time. Let us briefly review two of these special substitutions in
our setting. As noted above, the substitution $u=v=0$ gives the number
of spanning hyperforests. 
On the special hyperbola $uv=1$ we have
$$
R(\sigma,\alpha;v^{-1},v)=\sum_{\beta\leq\alpha}
v^{-\kappa(\sigma,\beta)+\kappa(\sigma,\alpha)}\cdot
v^{\kappa(\sigma,\beta)+n-z(\beta)-z(\sigma)}
=\sum_{\beta\leq\alpha}
v^{n+\kappa(\sigma,\alpha)-z(\beta)-z(\sigma)}
$$
which amounts to a weighted counting of the refinements of $\alpha$ by
the number of cycles. It is not likely that the evaluation in all
remaining special cases would be equally easy for hypermaps for two
reasons. First, the situation is more complex for matroids in general
(for details see~\cite{Jaeger}), similar complications may be reasonably
expected for hypermaps. Second, defining the Tutte polynomial
$T(\sigma,\alpha;x,y)$ as $R(\sigma,\alpha,x-1,y-1)$
for collections of hypermaps does not seem to be a good idea 
because of the following example.
\begin{example}
\label{ex:Narayana}  
Consider the hypermap $(\sigma,\alpha)$ given by $\sigma=(1)(2)\cdots(n)$
and $\alpha=(1,2,\ldots,n)$. For any $\beta$ refining $\alpha$ we
have $\kappa(\sigma,\beta)=z(\beta)$ and the exponent of $v$ in
$R(\sigma,\alpha;u,v)$ is 
$$\kappa(\sigma,\beta)+n-z(\beta)-z(\sigma)=z(\beta)+n-z(\beta)-n=0
$$
The exponent of $u$ is $z(\beta)-1$ and we obtain the {\em Narayana
  polynomial} of $u$, that is, the generating function of the
noncrossing partitions of $\{1,2,\ldots,n\}$, weighted by the number of
parts. We obtain the same polynomial, but of $v$, if we compute the
Whitney polynomial of the dual hypermap
$(\alpha^{-1}\sigma,\alpha^{-1})$. For $n=2$ these Narayana
polynomials are $1+u$ and $1+v$ respectively, which meshes well with the
recurrences for the Whitney polynomials and Tutte
polynomials of maps. It is doubtful that a similar substitution rule
would work well for longer cycles of $\alpha$. For example, for $n=3$ we get
$R(\sigma,\alpha;u,v)=u^2+3u+1$, substituting $u=x-1$ yields $x^2+x-1$,
a polynomial with a negative coefficient.  
\end{example}  
On the other hand, Example~\ref{ex:Narayana} inspires the following
result, generalizing the well-known formula for Tutte polynomials of
planar maps.  
\begin{theorem}
\label{thm:planardual}   
Exchanging the two variables in the Whitney polynomial of collection of
hypermaps $(\sigma,\alpha)$ of genus zero yields the Whitney polynomial
of its dual $(\alpha^{-1}\sigma,\alpha^{-1})$:
$$
R(\sigma,\alpha;u,v)=R(\alpha^{-1}\sigma,\alpha^{-1};v,u).
$$
\end{theorem}  
\begin{proof}
Let $(\sigma,\alpha)$ be a collection of hypermaps of genus zero on
$\{1,2,\ldots,n\}$, and let $\beta$ be any permutation
$\{1,2,\ldots,n\}$ refining $\alpha$. It suffices to show that the
contribution $w(\beta)$ of $\beta$ to $R(\sigma,\alpha;u,v)$ is the same
as the contribution $w^*(\alpha^{-1}\beta)$ of $\alpha^{-1}\beta$ to
$R(\alpha^{-1}\sigma,\alpha^{-1};v,u)$. Indeed, the map $\beta\mapsto
\alpha^{-1}\beta$ is a bijection from the set of permutations refining
$\alpha$ to the set of permutations refining $\alpha^{-1}$: if $\beta_i$
is the permutation replacing the cycle $\alpha_i$ of $\alpha$ then
$g(\alpha_i,\beta_i)=0$ is easily seen to be equivalent to
$g(\alpha_i^{-1}, \alpha_i^{-1}\beta_i)=0$.

The collection of hypermaps
$(\sigma,\beta)$ has genus zero, hence we have:
$$
n+2\kappa(\sigma,\beta)-z(\beta)-z(\sigma)-z(\beta^{-1}\sigma)=0.
$$
Solving this equation for $\kappa(\sigma,\beta)$ we obtain:
\begin{equation}
\label{eq:kbeta}  
\kappa(\sigma,\beta)=\frac{z(\beta)+z(\sigma)+z(\beta^{-1}\sigma)-n}{2}.
\end{equation}  
Similarly $\kappa(\sigma,\alpha)$ may be written as 
\begin{equation}
\label{eq:kalpha}  
\kappa(\sigma,\alpha)=\frac{z(\alpha)+z(\sigma)+z(\alpha^{-1}\sigma)-n}{2}.
\end{equation}  
Taking the difference of~(\ref{eq:kbeta}) and~(\ref{eq:kalpha}) we
obtain that the exponent of $u$ in $w(b)$ is
    \begin{equation}
      \label{eq:wbeta1}
\kappa(\sigma,\beta)-\kappa(\sigma,\alpha)
      =\frac{z(\beta)-z(\alpha)+z(\beta^{-1}\sigma)-z(\alpha^{-1}\sigma)}{2}. 
     \end{equation} 
Since $\beta$ is a refinement of $\alpha$, the collection of hypermaps
$(\alpha,\beta)$ has genus zero and $z(\alpha)$ components. By the
definition of the genus we have 
$$
n+2z(\alpha)-z(\alpha)-z(\beta)-z(\alpha^{-1}\beta)=0.
$$
Using this equation we may rewrite $z(\beta)-z(\alpha)$ as
$$
z(\alpha)-z(\beta)=z(\alpha^{-1}\beta)-n.
$$
Substituting this into~(\ref{eq:wbeta1}) we obtain that the exponent of
$u$ in $w(b)$ is 
    \begin{equation}
      \label{eq:uwbeta}
\kappa(\sigma,\beta)-\kappa(\sigma,\alpha)
 =\frac{n-z(\alpha^{-1}\beta)-z(\alpha^{-1}\sigma)+z(\beta^{-1}\sigma)}{2}.  
     \end{equation} 
On the other hand, by~(\ref{eq:kbeta}), the exponent of $v$ in
$w(\beta)$ is
    \begin{equation}
      \label{eq:vwbeta}
      \kappa(\sigma,\beta)+n-z(\beta)-z(\sigma)=
      \frac{n-z(\beta)-z(\sigma)+z(\beta^{-1}\sigma)}{2}.
      \end{equation}
Note that the right hand side of~(\ref{eq:uwbeta}) may be obtained from
the right hand side of~(\ref{eq:vwbeta}) by replacing $\beta$ with
$\alpha^{-1}\beta$ and $\sigma$ with $\alpha^{-1}\sigma$. (The only not
completely trivial part of the verification is to observe that
$\beta^{-1}\sigma=(\alpha^{-1}\beta)^{-1}(\alpha^{-1}\sigma)$.) Hence
the exponent of $v$ in $w(\beta)$ is the same as the exponent of $v$ in
$w^{*}(\alpha^{-1}\beta)$. The analogous statement for the exponents of
$u$ follows immediately after observing that the operation $\beta\mapsto
\alpha^{-1}\beta$ is an involution. 
\end{proof}

\section{The medial map of a hypermap}
\label{sec:medial}

\begin{definition}
Let $(\sigma,\alpha)$ be a hypermap on the set of points
$\{1,2,\ldots,n\}$. We define its {\em medial map $M(\sigma,\alpha)$} as
the following map $(\sigma',\alpha')$:
\begin{itemize}
\item[-] The set of points of $(\sigma',\alpha')$ is
  $\{1^-,1^+,2^-,2^+,\ldots,n^-,n^+\}$;
\item[-] the cycles of $\sigma'$ are all cycles of the form
  $(i_1^-,i_1^+,i_2^-,i_2^+,\ldots,i_k^-,i_k^+)$ where\\
  $(i_1,i_2,\ldots,i_k)$ is a cycle of $\alpha$;
\item[-] the cycles of $\alpha'$ are all cycles of the form
  $(i^+,\sigma(i)^-)$.  
\end{itemize}
We extend the definition to collections of hypermaps using the same rules.
\end{definition}  
The medial map of a hypermap is always an {\em Eulerian map}, defined as
follows. 
\begin{definition}
An {\em Eulerian map} is a map $(\sigma,\alpha)$ on the set of points\\
$\{1^-,1^+,2^-,2^+,\ldots,n^-,n^+\}$ that has the following property: each
cycle of $\sigma$ is of the form
$(i_1^-,i_1^+,i_2^-,i_2^+,\ldots,i_k^-,i_k^+)$ and each edge of $\alpha$
connects a pair of points of opposite signs. We extend 
the definition to collections of maps the same way. We define the {\em
  underlying Eulerian digraph} of a collection of Eulerian maps is obtained by
directing each edge $(i^+,j^-)\in \alpha$ from its positive endpoints
toward its negative endpoint. 
\end{definition}  

\begin{figure}[h]
\begin{center}
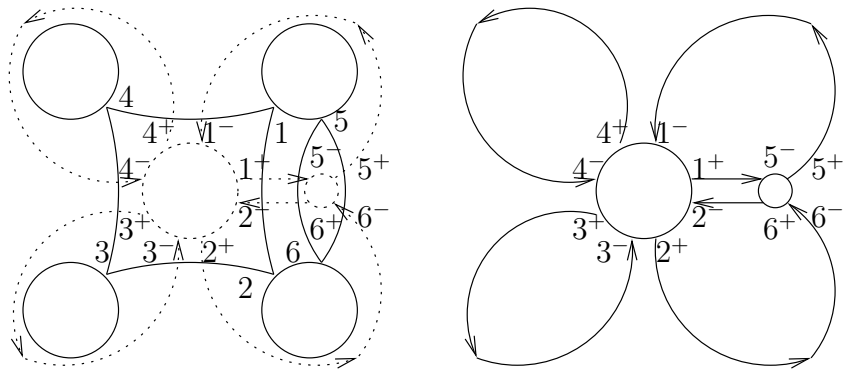
\end{center}
\caption{A hypermap and its medial map}
\label{fig:medial} 
\end{figure}

The process of creating the medial map is shown in
Figure~\ref{fig:medial}. In our example we construct the medial map of
the planar hypermap $((1,5)(2,6)(3)(4),(1,2,3,4)(5,6))$. The hyperedges
$(1,2,3,4)$ and $(5,6)$ give rise to the vertices
$(1^-,1^+,2^-,2^+,3^-,3^+,4^-,4^+)$ and $(5^-,5^+,6^-,6^+)$. We should
think of the point $i^-$, respectively $i^+$ immediately preceding,
respectively succeeding the point $i$ in the cyclic order. Next we use
the vertices of the hypermap to define the edges: the cycle $(1,5)$
gives rise to the edges $(1^+,5^-)$ and $(5^+,1^-)$. Note that, for any
planar hypermap we may first replace the cycle
$(j_1,j_2,\ldots,j_{\ell})$ of $\sigma$ with the cycle
$$j_1^-\mapsto j_1^+ \mapsto j_2^-\mapsto j_2^+\mapsto
,\ldots,j_{\ell}^-\mapsto j_{\ell}^+ \mapsto j_1^- ,$$
then we erase the arrows $j_s^{-}\mapsto j_s^+$, and the
remaining arrows form the edges of the medial hypermap. By definition,
the edges of a map are {\em not} oriented, but we use the signature of
the points to define an orientation: we set all arrows to have their tail
at their positive end and their head at their negative end. We call the
resulting underlying digraph the {\em directed medial graph of the
  hypermap}. This is always an Eulerian digraph as the indegree of each
vertex equals its outdegree. The converse is also true:
\begin{proposition}
\label{prop:dmg}  
Every directed Eulerian graph arises as the
directed medial graph of a collection of hypermaps.
\end{proposition}  
\begin{proof}
Consider an Eulerian digraph
$D(V,E)$. Let us define the set of points as the set of all pairs
$(v,e)^+$ where $v$ is the tail of $e$ together with the set of all
pairs $(v,e)^-$ where $v$ is the head of $e$. (Loop arrows give rise to
two points.) Let $2n$ be the number of resulting points.
At each vertex $v$ list the points $(v,e)^+$ and $(v,e)^-$
in cyclic order in such a way that positive and negative points
alternate. This is possible since the graph is a directed Eulerian
graph. We relabel the points with the elements of the set 
$\{1^-,1^+,2^-,2^+,\ldots,n^-,n^+\}$ in such a way that we don't change their
signs, we use each label exactly once, and at each vertex we obtain a
cycle of the form $(i_1^-,i_1^+,i_2^-,i_2^+,\ldots,i_k^-,i_k^+)$. It may
be shown by induction on the number of vertices that such a relabeling
is possible. We define the vertex permutation $\sigma$ as the
permutation obtained by replacing each cycle
$(i_1^-,i_1^+,i_2^-,i_2^+,\ldots,i_k^-,i_k^+)$ with the cycle
$(i_1,i_2,\ldots,i_k)$. Clearly $\sigma$ is a permutation of
$\{1,2,\ldots,n\}$. Finally we define the permutation $\alpha$ by
setting $\alpha(i)=j$ exactly when there is an edge containing $i^+$ and
$j^-$. It is easy to see that the medial digraph of $(\sigma,\alpha)$ is
isomorphic to the Eulerian digraph we started with. 
\end{proof}  
Observe that definition of the underlying medial digraph of a genus zero
map essentially coincides with the existing definition of the medial digraph
of a planar graph.
\begin{figure}[h]
\begin{center}
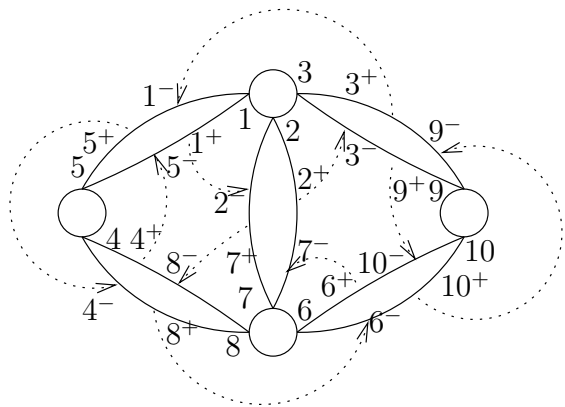
\end{center}
\caption{The medial map of a planar map}
\label{fig:mmap} 
\end{figure}
An example is shown in Figure~\ref{fig:mmap}. Here we see the genus zero
map $(\sigma,\alpha)$ with $\sigma=(1,2,3)(4,5)(6,7,8)(9,10)$ and
$\alpha=(1,5)(2,7)(3,9)(4,8)(6,10)$. The edges of the medial map are
represented as dashed arcs. The convention of directing each edge of the
medial map from its positive end towards its negative end coincides with
the requirement that the shortest cycles of the medial map surrounding a
vertex of the map $(\sigma,\alpha)$ should be oriented counterclockwise.

Although most of our results are about planar hypermaps, we wish to
point out that the above definitions are valid for all hypermaps and
preserve the genus.
\begin{proposition}
Let $(\sigma,\alpha)$ be any hypermap on the set of points
$\{1,2,\ldots,n\}$. Then its medial map $M(\sigma,\alpha)$ is a map and
its genus $g(M(\sigma,\alpha))$ equals the genus $g(\sigma,\alpha)$  of
$(\sigma,\alpha)$. 
\end{proposition}  
\begin{proof}
First we check that the medial map is connected: using the cycles of
$\sigma'$ we may reach the positive copy of every point, hence it
suffices to check that any $j^+$ may be reached from any $i^+$. Consider
any directed path $i=i_1\rightarrow i_2\rightarrow \cdots \rightarrow
i_k=j$ in which each step corresponds to $i_s\rightarrow \sigma(i_s)=i_{s+1}$ or
$i_s\rightarrow \alpha(i_s)=i_{s+1}$. If $i_{s+1}=\sigma(i_s)$ then
$$i_{s+1}^+=\sigma(i_s)^+=\sigma'(\sigma(i_s)^-)=\sigma'\alpha'(i_s^+),$$
and if $i_{s+1}=\alpha(i_s)$ then
$$i_{s+1}^+=\alpha(i_s)^+=\sigma'(\alpha(i_s)^-)=\sigma'^{2}(i_s^+).$$

The cycles of the permutation $\sigma'$ are obtained by doubling the
number of points on each cycle of $\alpha$, hence we have
$z(\sigma')=z(\alpha)$. The cycles of $\alpha'$ form a matching on $2n$
points, hence we have $z(\alpha')=n$. Finally, for $\alpha'^{-1}\sigma'$
we have
\begin{align*}
\alpha'^{-1}\sigma'(i^-)&=\alpha'^{-1}(\alpha^{-1}(i)^+)=\sigma\alpha^{-1}(i)^-\quad\mbox{and}\\
\alpha'^{-1}\sigma'(i^+)&=\alpha'^{-1}(i^+)=\sigma^{-1}(i)^-.\\
\end{align*}
Hence $\alpha'^{-1}\sigma'$ has
$z(\sigma\alpha^{-1})=z(\alpha^{-1}(\sigma\alpha^{-1})\alpha)=z(\alpha^{-1}\sigma)$
cycles on the negative points and $z(\sigma^{-1})=z(\sigma)$ cycles on
the positive points. The genus formula gives
\begin{align*}
g(M(\sigma,\alpha))&=2n+2-z(\sigma')-z(\alpha')-z(\alpha'^{-1}\sigma')\\
&=2n+2-z(\alpha)-n-(z(\alpha^{-1}\sigma)+z(\sigma))\\
&=n+2-z(\alpha)-z(\alpha^{-1}\sigma)-z(\sigma)=2g(\sigma,\alpha).
\end{align*}
\end{proof}  

Next we generalize circuit partitions in such a way that
(\ref{eq:MartinT}) still holds.

\begin{definition}
Let $(\sigma,\alpha)$ be a collection of Eulerian maps. A {\em
  noncrossing Eulerian state} is a partitioning of the edges of the
underlying directed medial graph into closed paths in such a way that
these paths do not cross at any of the vertices.   
\end{definition}  

\begin{figure}[h]
\begin{center}
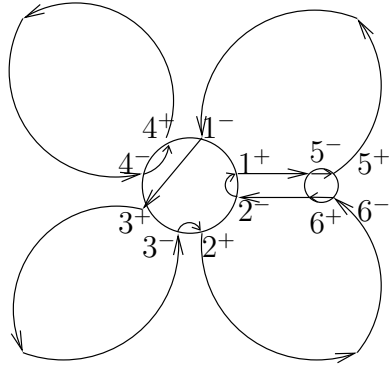
\end{center}
\caption{A noncrossing Eulerian state}
\label{fig:ncccp} 
\end{figure}
Figure~\ref{fig:ncccp} represents a noncrossing Eulerian state of the
Eulerian map shown in the right hand side of Figure~\ref{fig:medial}. We
partition the set of edges into closed paths by matching each negative
point on a vertex to a positive point on the same vertex. While the
edges of the underlying Eulerian digraph are directed from the positive
points towards the negative points, the arrows inside the vertices point
from the negative points towards the positive points. We call the
resulting Eulerian state noncrossing if the arrows inside the vertices
can be drawn in a way that they do not cross, equivalently, inside each
vertex, the matching induced by the arrows must be a noncrossing
partition. Note that in the medial map of a map each vertex has four
points, two negative and two positive, the positive and the negative
points alternate, there are exactly two ways to match each positive
point to a negative point in a vertex, and both ways yield a noncrossing
matching. Thus for Eulerian maps that are medial maps of planar maps the
definition of a noncrossing Eulerian state coincides with the usual definition 
of an Eulerian state. In the general case, a noncrossing Eulerian state
is uniquely defined by a matching that refines the vertex permutation of
the Eulerian map, and matches positive points to negative points. We
call such a matching a {\em coherent matching} of the Eulerian map. 

\begin{definition}
We define the {\em noncrossing circuit partition
polynomial} of an Eulerian map $(\sigma,\alpha)$ as 
$$
  j((\sigma,\alpha);x)=\sum_{k\geq 0} f_k(\sigma,\alpha) x^k.
  $$
Here $f_k(\sigma,\alpha)$ is the number of noncrossing Eulerian states
with $k$ cycles.   
\end{definition}

\begin{remark}
\label{rem:maps-nc}
If $(\sigma,\alpha)$ is a map then the above noncrossing condition is
automatically satisfied by all Eulerian states of the medial map 
$M(\sigma,\alpha)$. Indeed, each vertex of $M(\sigma,\alpha)$
corresponds to a $2$-cycle $(i,j)$ of $\alpha$, which gives rise to a
$4$-cycle $(i^+,i^-,j^+,j^-)$. To create an Eulerian state, we must
match points of opposite sign: either we match $i^+$ with $i^-$ and
$j^+$ with $j^-$, or we match $i^+$ with $j^-$ and
$j^+$ with $i^-$. Both matchings are coherent.      
\end{remark}

\begin{theorem}
\label{thm:genmartin}
Let $(\sigma,\alpha)$ a genus zero collection of hypermaps and
$M(\sigma,\alpha)$ the collection of its medial maps. Then
$$
j(M(\sigma,\alpha);x)=x^{\kappa(\sigma,\alpha)} R(\sigma,\alpha;x,x)
$$
\end{theorem}  
\begin{proof}
Direct substitution into the definition of $R(\sigma,\alpha;u,v)$ yields
$$
R(\sigma,\alpha;x,x)=\sum_{\beta\leq\alpha}
x^{2\kappa(\sigma,\beta)+n-z(\beta)-z(\sigma)-\kappa(\sigma,\alpha)}. 
$$
Each connected component of $(\sigma,\beta)$ is a hypermap
$(\sigma_i,\beta_i)$ on $n_i$ points where $i=1,2,\ldots,\kappa(\sigma,\beta)$
and $n_1+n_2+\cdots+n_{\kappa(\sigma,\beta)}=n$. Since each
$(\sigma_i,\beta_i)$ has genus zero, we have
$$
n_i+2-z(\sigma_i)-z(\beta_i)-z(\beta_i^{-1}\sigma)=0
$$
Summation over $i$ and rearranging yields
$$
2\kappa(\sigma,\beta)+n-z(\beta)-z(\sigma)=z(\beta^{-1}\sigma).
$$
Hence we have
$$
x^{\kappa(\sigma,\alpha)}R(\sigma,\alpha;x,x)=\sum_{\beta\leq\alpha}
x^{z(\beta^{-1}\sigma)}.  
$$
Let $\mu$ be an coherent matching associated to the collection of medial
maps $M(\sigma,\alpha)$. We associate to $\mu$ the following refinement
$\beta$ of $\alpha$: if $\mu$ matches $i^+$ to $j^-$ then $\beta$ sends $i$ to
$j$. It is easy to see that  correspondence $\mu\mapsto \beta$ is a
bijection. Observe finally that the number of circuits of the
noncrossing Eulerian state associated to $\mu$ is $z(\beta^{-1}\sigma)$.  
For example, the noncrossing Eulerian state shown in Figure~\ref{fig:ncccp} is
induced by the coherent matching
$(1^+,2^-)(2^+,3^-)(3^+,1^-)(4^+,4^-)(5^+,5^-)(6^+,6^-)$ which
corresponds to the refinement $\beta=(1,2,3)(4)(5)(6)$ of $\alpha$. The
cycles of the permutation $\beta^{-1}\sigma=(1,5,3,2,6)(4)$  correspond
to the cycles $(1^+,5^-,5^+,1^-,3^+,3^-,2+,6^-,6+,2^-)$ and $(4^+,4^-)$
of the Eulerian state induced by $\mu$. In each cycle of the noncrossing
Eulerian state each negative point $i^-$ is followed by the positive
point $\beta^{-1}(i)^+$ and each positive point $i^+$ is followed by the
negative point $\sigma(i)^-$. Hence parsing the cycles of the
noncrossing Eulerian state amounts to a visual computation of
$\beta^{-1}\sigma$.    
\end{proof}

\section{A visual approach to computing the Whitney polynomial
  of a planar hypermap}
\label{sec:visual}

In this section we introduce a visual model based on a slight refinement of the
results~(\ref{eq:MartinT}) and (\ref{eq:cppT}) connecting the Tutte
polynomial of a planar map and the circuit partition polynomial of its
directed medial map. 

\begin{figure}[h]
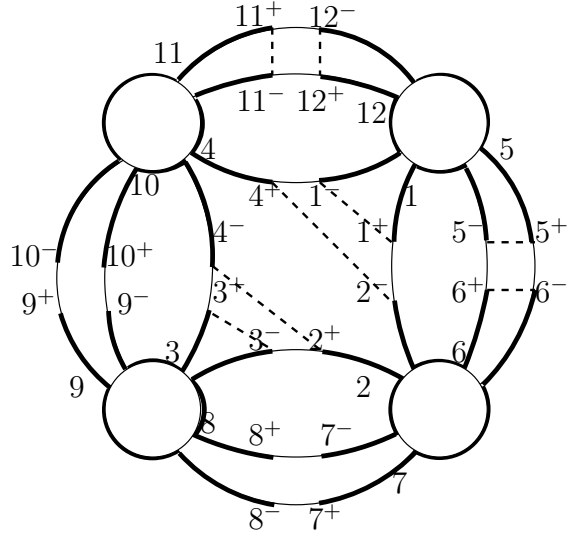
\caption{A planar hypermap with the edges of its medial map shrunk to
  its outline}
\label{fig:papercuts}
\end{figure}
Consider the planar hypermap $(\sigma,\alpha)$ with 
$\sigma=(1,5,12)(4,11,10)(3,9,8)(2,7,6)$ and
$\alpha=(1,2,3,4)(5,6)(7,8)(9,10)(11,12)$, represented in
Figure~\ref{fig:papercuts}.  For each point $i$ we also added the points
$i^-$ and $i^+$ that are used to construct the medial map. The vertices
of the medial map are the hyperedges of the original map, and we
``shrink'' the edges $(i^+,\sigma(i)^-)$ of the medial map to follow
the outline of the original vertices and the hyperedges. For example,
$\sigma(1)=5$ and the edge $(1^+,5^-)$ of the medial map is represented
by the thick curve that begins at $1^+$, follows the outline of the
hyperedge to the vertex $(1,5)$, there it follows the outline of the
vertex $(1,5)$ towards the edge $(5,6)$ and finally it follows the
outline of the edge  $(5,6)$ to the point $5^-$. (There is no need to
direct the curves, 
orientation is defined by the fact that vertices are oriented
counterclockwise and hyperedges are oriented clockwise.) 
Adding such a thick curve for each edge
of the medial map results in fattening the outline of the diagram of
$(\sigma, \alpha)$, except for the following arcs:
\begin{enumerate}
\item The arcs on the vertices corresponding going from $i^-$ to
  $i^+$ counterclockwise are not 
  thick, except for the special case when $\sigma(i)=i$ (not shown in
  this example). In the typical situation the point $i^-$ is connected
  to $\sigma^{-1}(i)^+$ via a thick curve, the point $i^+$ is
  connected to $\sigma(i)^-$ via a thick curve. 
\item The arcs on the hyperedges representing the possibility of going
  from $i^+$ to $\alpha(i)^{-}$ are not thick, except for the special
  case when $i$ is a {\em bud}, that is, a fixed point of $\alpha$. Buds
  can be added to or removed from the picture without changing anything
  of substance.     
\end{enumerate}
Next we select a coherent matching on the signed points. Matching $i^+$
to $\alpha(i)^{-}$ is always a valid choice (and this is the only choice
for buds, which we may ignore), this selection does not
interfere with the rest. If we make this selection for all $i$, we simply
fill the gaps indicated by thin lines along the hyperedges, the circuits of our
circuit partition will simply represent all cycles of $\alpha^{-1}\sigma$,
that is, the faces  of $(\sigma,\alpha)$. A {\em nontrivial choice}  is
to match some $i^+$ to some point different from $\alpha(i)^-$. In
Figure~\ref{fig:papercuts} the nontrivial choices are represented by
dashed lines: $11^+$ is matched to $11^-$, $12^+$ is matched to $12^-$,
$4^+$  is matched to $2^-$ and so on. We can think of the diagram of
$(\sigma,\alpha)$ as a paper cutout, with the vertices and the
hyperedges being solid and the faces missing. Each nontrivial pair of
matched points corresponds then to a cut into the object using a
scissor, subject to the following rules:
\begin{itemize}
\item[(R1)] Each cut is a simple curve connecting a point $i^+$ with a
  point $j^-$, inside a hyperedge. (In particular, we are not allowed to
  cut into the vertex areas.) 
\item[(R2)] Each point $i^+$ and $j^-$ may be used at most once. (Giving a
  partial matching.)
\item[(R3)] The remaining points not used in the cuts
  must come in pairs $(i^+,\alpha(i)^-)$. (Together with the trivial
  choices we have a matching.) 
\item[(R4)] A new cut cannot cut into the cut-line of a previous cut (the
  matching needs to be coherent).   
\end{itemize}  
Nontrivial cuts replace the set $\alpha$ of hyperedges with a refinement
$\beta$ of $\alpha$. At the end of the process the curves of the outline
correspond to the faces $\beta^{-1}\sigma$ of the collection of
hypermaps $(\sigma,\beta)$ and they are also the circuits of the
corresponding circuit partition. 

So far our construction offers a visual proof of
Theorem~\ref{thm:genmartin} (and of~(\ref{eq:MartinT}) and
(\ref{eq:cppT})). Let us refine the picture now by counting connected
components visually as well. For a planar map, it is not unusual to think
of its unbounded face as special. Let us think of the unbounded face of
a hypermap $(\sigma,\alpha)$ as ``the ocean''  which ``makes its
coastline wet''. After making a few nontrivial cuts, the resulting collection of
hypermaps may have several connected components, each has one
coastline. Figure~\ref{fig:papercuts2} illustrates this situation after
performing the nontrivial cuts indicated in
Figure~\ref{fig:papercuts}. The shaded regions indicate the faces of
$\sigma,\alpha$ whose border has been merged with a wet coastline. The
wet coastlines are thickened.
\begin{figure}[h]
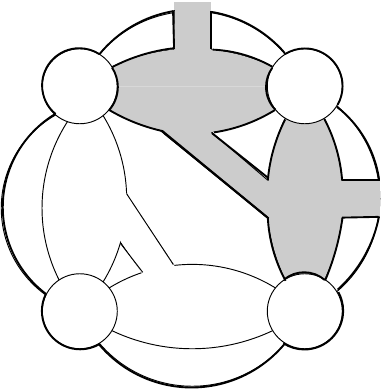
\label{fig:papercuts2}
\end{figure}
To summarize, we obtain the following.
\begin{theorem}
\label{thm:visual}
Given a planar hypermap $(\sigma,\alpha)$, we may visually compute its
Whitney polynomial by making its model in paper, and performing the
cutting procedures subject to the rules (R1), (R2), (R3) and (R4) in all
possible ways and associating to each outcome $u$ raised to the power of
the wet coastlines and $v$ raised to the power of the dry faces. The sum
of all weights is $u\cdot R(\sigma,\alpha;u,v)$.  
\end{theorem}  

\section{Counting the noncrossing Eulerian colorings of the directed medial
  map of a planar hypermap}
\label{sec:medialc}

In this section we extend the formula counting the Eulerian colorings of
the medial graph of a plane graph~\cite[Evaluation
  6.9]{Ellis-Monaghan-exploring} to medial maps of planar hypermaps. The
following definition may be found in~\cite[Definition
  4.3]{Ellis-Monaghan-exploring}. 

\begin{definition} 
An {\em Eulerian $m$-coloring} of an Eulerian directed graph
$\overrightarrow{G}$ is an edge coloring
of $\overrightarrow{G}$ with $m$ colors so that for each color the
(possibly empty) set of all edges of the given color forms an Eulerian
subdigraph.  
\end{definition}

Consider now a planar hypermap $(\sigma,\alpha)$ and its directed
Eulerian medial map $M(\sigma,\alpha)$. Given an Eulerian $m$-coloring
of the edges, let us color the endpoints of $(i^+,\sigma(i)^-)$ with the
color of the edge. We call this coloring of the points the {\em coloring
  of the points induced by the Eulerian $m$-coloring}. In order to
relate the count of the Eulerian $m$-colorings to our Whitney
polynomial, we must restrict our attention to {\em noncrossing Eulerian
  $m$-colorings}, defined as follows.
\begin{definition}
Let  $(\sigma,\alpha)$ be a planar hypermap and let $M(\sigma,\alpha)$
be its directed medial map. We call an Eulerian $m$-coloring {\em
  noncrossing} if there is a noncrossing Eulerian state 
such that all edges of the same connected circuit have the same color.  
\end{definition}  
\begin{remark}
If $(\sigma,\alpha)$ is a map then the above noncrossing condition is
automatically satisfied by each Eulerian $m$-coloring of
$M(\sigma,\alpha)$: as noted in Remark~\ref{rem:maps-nc}, all Eulerian
states of $M(\sigma,\alpha)$ are noncrossing. The more general case
of partitioning the edge set of an Eulerian digraph into Eulerian
subdigraphs was addressed in the work of Bollob\'as and
Ellis-Monaghan~\cite{Arratia-Bollobas,Bollobas-cp,Ellis-Monaghan-mp}.      
\end{remark}

Using the induced coloring of the points we can verify vertex by vertex
whether an Eulerian $m$-coloring is non-crossing: it is necessary and
sufficient to be able to find a coherent matching at each vertex such
that only points of the same color are matched. This observation
motivates the following definition.

\begin{definition}
\label{def:legalc}  
Let $(\sigma,\alpha)$ be a planar hypermap on the set of points
$\{1,2,\ldots,n\}$ and $M(\sigma,\alpha)=(\sigma',\alpha')$ its directed
medial map on $\{1^-,1^+,2^-,2^+,\ldots,n^-,n^+\}$. We
call an $m$-coloring of the points $\{1^-,1^+,2^-,2^+,\ldots,n^-,n^+\}$
a {\em legal coloring} if it satisfies the following conditions:
\begin{enumerate}
\item The endpoints of each edge $(i^+,\sigma(i)^-)\in \alpha'$ of
  $M(\sigma,\alpha)$ have the same color.
\item There is a coherent matching of the set
  $\{1^-,1^+,2^-,2^+,\ldots,n^-,n^+\}$ such that each point is matched
  to a point of the same color. 
\end{enumerate}
\end{definition}  

Definition~\ref{def:legalc} is motivated by the following observation. 
\begin{proposition}
Given a planar hypermap $(\sigma,\alpha)$ and its directed medial map
$M(\sigma,\alpha)$ a coloring of the set of points of $M(\sigma,\alpha)$
is legal if and only if it is induced by a noncrossing Eulerian
$m$-coloring of the edges of $M(\sigma,\alpha)$.   
\end{proposition}  
\begin{proof}
Condition (1) requires that the coloring of the signed points must be induced
by a coloring of the edges. This coloring of the signed points is
induced by an Eulerian coloring if and only if each cycle of $\sigma'$
contains the same number of positive points and negative points of each
color. The resulting partitioning of the edges into color sets is
represented by a noncrossing Eulerian state if and only if there is a
coherent matching connecting only points of the same color.        
\end{proof}  

Note that for a given $m$-coloring of the signed points, induced by a
coloring of the edges in $\alpha'$, condition~(2) may be independently
verified at each vertex of $\sigma'$. This observation motivates the
following definition.
\begin{definition}
Let $(i_1^-,i_1^+,i_2^-,i_2^+,\ldots,i_k^-,i_k^+)$ be a cyclic signed
permutation and let is fix an $m$-coloring of its points. We say that
the {\em valence} of this colored cycle is number of coherent matchings
of its points that match each point the a point of the same color.   
\end{definition}  

Now we are able to state the generalization of~\cite[Evaluation
  6.9]{Ellis-Monaghan-exploring}.

\begin{theorem}
Let $(\sigma,\alpha)$ be a planar hypermap. Then, for a fixed positive
integer $m$, we have
$$
m^{\kappa(\sigma,\alpha)} R(\sigma,\alpha;m,m)=\sum_{\lambda}
\prod_{v\in\sigma'} \nu(v,\lambda).
$$
Here the summation runs over all Eulerian $m$-colorings $\lambda$ of the
directed medial map $M(\sigma,\alpha)=(\sigma',\alpha')$, and for each vertex
$v\in\sigma'$ the symbol $\nu(v,\lambda)$ represents the valence of $v$
colored by the restriction of the point coloring induced by $\lambda$ to the
points of $v$.  
\end{theorem}  
\begin{proof}
By Theorem~\ref{thm:genmartin} the left hand side is
$j(M(\sigma,\alpha);m)$ which is the number of ways to select a
noncrossing Eulerian state, and then color each of its cycles with one
of $m$ colors. By the above reasoning, this may be also performed by
first selecting an Eulerian $m$-coloring $\lambda$ of the
directed medial map $M(\sigma,\alpha)$ and then a coherent matching on
each vertex that matches only points of the same color.  
\end{proof}  

\begin{example}
For maps $(\sigma,\alpha)$, there are essentially two types of vertices
in the directed medial map $M(\sigma,\alpha)$: monochromatic vertices
and vertices colored with two colors. (We cannot have more than two
colors because four points form two matched pairs.) For monochromatic
vertices there are $2$ coherent matchings, for vertices colored with two
colors we have only one way to match the points of the same color. This
gives exactly the formula~\cite[Evaluation
  6.9]{Ellis-Monaghan-exploring}:
$$
m^{\kappa(\sigma,\alpha)}R(\sigma,\alpha;m,m)=\sum_{\lambda} 2^{m(\lambda)}
$$
where $m(\lambda)$ is the number of monochromatic vertices. Evaluating
this formula at $m=2$ was used by Las Vergnas~\cite{LasVergnas-tp} to
describe the exact power of $2$ that divides $R(\sigma,\alpha;2,2)$ for
a map $(\sigma,\alpha)$, or equivalently the evaluation of its Tutte
polynomial at $(3,3)$.    
\end{example}

\section{The characteristic polynomial of a hypermap}
\label{sec:charpoly}

In this section we generalize the notion of a characteristic polynomial
from graphs and graded partially ordered sets to hypermaps. We begin with
having a closer look at the partially ordered set of refinements. 

\begin{proposition}
Let $\alpha$ be a permutation of $\{1,2,\ldots,n\}$ with $k$ cycles and
let $c_1,\ldots,c_k$ 
be the lengths of these cycles. Then the partially ordered set of all
refinements of $\alpha$, ordered by the refinement operation, is the
direct product
$$
[\id,\alpha]=\prod_{i=1}^k \operatorname{NC}(c_i).
$$
Here $\id$ is the identity permutation $(1)(2)\cdots(n)$ and 
$\operatorname{NC}(c_i)$ is the lattice of noncrossing partitions
on $c_i$ elements.
\end{proposition}  
The M\"obius function of the noncrossing partition lattice is known to
be
$$
\mu(\operatorname{NC}(n))=(-1)^{n-1} C_n
$$
where $C_n$ is the $n$th Catalan number~\cite[Formula (2)]{Simion-noncrossing}.
Using this result, the M\"obius function of any interval $[\beta,\alpha]$ may be
expressed using the multiplicativity of the M\"obius function and the
observation that each interval in the noncrossing partition lattice is
also a product of intervals~\cite[p.\ 401]{Simion-noncrossing}. We will use the
notation $\mu(\beta,\alpha)$ to denote the M\"obius function of the
interval $[\beta,\alpha]$ ordered by refinement. As explained above,
$\mu(\beta,\alpha)$ is always a signed product of Catalan numbers.
\begin{remark}
\label{rem:maps}
When $(\sigma,\alpha)$ is a collection of maps then for each refinement
$\beta\leq 
\alpha$ we have $\mu(\beta,\alpha)=(-1)^{z(\beta)-z(\alpha)}$.
\end{remark}  

\begin{definition}
\label{def_char}
Given a collection of hypermaps $(\sigma,\alpha)$ on the set of
points $\{1,2,\ldots,n\}$, we define its {\em
  characteristic polynomial $\chi(\sigma,\alpha;t)$} by 
$$\chi(\sigma,\alpha;t)=
\sum_{\beta\leq
    \alpha}\mu(\id,\beta)\cdot
t^{\kappa(\sigma,\beta)-\kappa(\sigma,\alpha)}. 
$$ 
\end{definition}
\begin{example}
When $(\sigma,\alpha)$ is a collection of maps on the set
$\{1,2,\ldots,n\}$, as noted in Remark~\ref{rem:maps}, we get  
$$
\chi(\sigma,\alpha;t)=
\sum_{\beta\leq
    \alpha}(-1)^{n-z(\beta)}\cdot t^{\kappa(\sigma,\beta)-\kappa(\sigma,\alpha)}
$$
which is exactly the characteristic polynomial of its underlying
graph. In this case we may also write
\begin{equation}
\label{eq:charpmap}  
\chi(\sigma,\alpha;t)=(-1)^{z(\sigma)-\kappa(\sigma,\alpha)}
R(\sigma,\alpha;-t,-1),  
\end{equation}
or, equivalently, 
$$
\chi(\sigma,\alpha;t)=(-1)^{z(\sigma)-\kappa(\sigma,\alpha)}
T(\sigma,\alpha;1-t,0),  
$$
where $T(\sigma,\alpha;x,y)$ is the Tutte polynomial of the map. 
\end{example} 

\begin{example}
If $\sigma=(1)(2)\cdots (n)$ and $\alpha=(1,2,\ldots,n)$ then the set of
refinements of $\alpha$ is the noncrossing partition lattice
$\operatorname{NC}(n)$, and
$\kappa(\sigma,\beta)-\kappa(\sigma,\alpha)=z(\beta)-1$ is the rank
$\rank(\beta)$ of the
noncrossing partition represented by $\beta$. The lattice
$\operatorname{NC}(n)$ is a graded partially ordered set and
$\chi(\sigma,\alpha;t)$ is exactly the characteristic polynomial
$$
\chi(\operatorname{NC}(n);t)=\sum_{\beta\leq \alpha} \mu(\id,\beta)
t^{\rank(\beta)} 
$$
as it is usually defined for such posets. It has been
shown by Jackson~\cite{Jackson} (a second proof was given by Dulucq and
Simion~\cite[Corollary 5.2]{Dulucq-Simion}) that for $n\geq 1$
\begin{equation}
\label{eq:charnc}  
\chi(\operatorname{NC}(n+1);t)=(-1)^{n} (1-t) R(\operatorname{NC}(n); (1-t))
\end{equation}
holds, where $R(\operatorname{NC}(n-1); t)$ is the rank generating
function of $\operatorname{NC}(n)$. This equation may look encouragingly
similar to (\ref{eq:charpmap}), but the careful reader should note the
shift of ranks in (\ref{eq:charnc}).
\end{example}  
M\"obius inversion works in this setting as follows. For a fixed
collection of hypermaps $(\sigma,\alpha)$, let us define the function
$X([\alpha_1,\alpha_2];t)$ on the intervals of the partially ordered set
of the refinements of $\alpha$ by
\begin{equation}
\label{eq:bigchi}  
X([\alpha_1,\alpha_2];t)=\sum_{\beta\in[\alpha_1,\alpha_2]}
  \mu(\alpha_1,\beta)\cdot
  t^{\kappa(\sigma,\beta)}
\end{equation}
By definition we have
\begin{equation}
\label{eq:charchr}
t^{\kappa(\sigma,\alpha)}\chi(\sigma,\alpha;t)=X([\id,\alpha];t)
\end{equation}
The usual M\"obius inversion formula computation yields
\begin{align*}
  \sum_{\beta\in [\alpha_1,\alpha_2]} X([\beta,\alpha_2];t)&=
  \sum_{\beta\in [\alpha_1,\alpha_2]}
  \sum_{\beta_1\in[\beta,\alpha_2]}
  \mu(\beta,\beta_1)\cdot
  t^{\kappa(\sigma,\beta_1)}\\
&=  \sum_{\beta\in [\alpha_1,\alpha_2]}
    t^{\kappa(\sigma,\beta_1)} 
    \sum_{\beta\in [\alpha_1,\beta_1]} \mu(\beta,\beta_1)\\
 &=  \sum_{\beta_1\in [\alpha_1,\alpha_2]}
    t^{\kappa(\sigma,\beta_1)}\cdot
    \delta_{\alpha_1,\beta_1} 
   =t^{\kappa(\sigma,\alpha_1)}.
\end{align*}
Here $\delta_{\alpha_1,\beta_1}$ is the Kronecker delta, hence we obtain
\begin{equation}
\label{eq:chromaticg}  
  \sum_{\beta\in [\alpha_1,\alpha_2]} X([\beta,\alpha_2];t)
  =
t^{\kappa(\sigma,\alpha_1)}.
\end{equation}  

Substituting $\alpha_1=\id$ and $\alpha_2=\alpha$ in
\eqref{eq:chromaticg} we get 
$$
\sum_{\beta\leq \alpha} X([\beta,\alpha];t)
=t^{\kappa(\sigma,\id)}.
$$
Using the fact that $\kappa(\sigma,\id)=z(\sigma)$ is the number of
vertices, we obtain 
\begin{equation}
\label{eq:chromatic}
\sum_{\beta\leq \alpha}
X([\beta,\alpha];t)=t^{z(\sigma)}. 
\end{equation}
For maps, Equation~(\ref{eq:chromatic}) is known to have the following
interpretation. Let us set $t=m$ for some positive integer $m$.
The number $m^{z(\sigma)}$ on the right hand side is then the number of ways to
color the vertices of $(\sigma,\alpha)$ using $m$ colors. For any such
coloring let $\beta$ be the refinement of $\alpha$ whose two-cycles are
exactly the edges that connect different vertices of the same color. Let
us call $\beta$ the {\em type} of the coloring. Using
Equation~(\ref{eq:chromaticg}) (whose derivation is simply
inclusion-exclusion in the case of maps) it is not hard to show that
$X([\beta,\alpha];t)$ is the the  
number of  all $n$-colorings of the vertices of $(\sigma,\alpha)$ that have type
$\beta$. The chromatic polynomial 
$m^{\kappa(\sigma,\alpha)}\chi(\sigma,\alpha;m)=X([\id,\alpha];m)$ is the
number of ways to color the vertices using $m$ colors such that adjacent
vertices have the same color.   
\begin{remark}
\label{rem:loop}  
If the map $(\sigma,\alpha)$ contains a loop edge, that is, a cycle
$(i,j)$ of $\alpha$ connects two points on the same cycle of $\sigma$
then each $\beta$ that is the type of a coloring of the vertices
contains $(i,j)$ and we have
$\kappa(\sigma,\alpha)\chi(\sigma,\alpha;t)=X([\id,\alpha];t)=0$. 
\end{remark}
We may extend the approach outlined above to collections of hypermaps
$(\sigma,\alpha)$ having hyperedges of length at most $3$. In this
generalization the following notion plays a key role.
\begin{definition}
Let $(\sigma,\alpha)$ be a collection of hypermaps and let $\alpha_1$ be
a refinement of $\alpha$. We say that a coloring of the cycles of
$\sigma$ is {\em $(\alpha_1,\alpha)$-compatible} if any pair of vertices
incident to the same cycle of $\alpha_1$ have the same color, and any
pair of vertices incident to the same cycle of $\alpha$ but not to the
same cycle of $\alpha_1$ have different colors.  
\end{definition}  

\begin{remark}
For collections of maps, selecting a refinement of $\alpha$ amounts
  to selecting a subset of edges. A vertex-coloring is
  $(\alpha_1,\alpha)$-compatible if the color is constant on each
  connected component of $(\sigma,\alpha_1)$, and vertices connected by
  an edge that belongs to $\alpha$ but by no edge belonging to
  $\alpha_1$ have different colors. The set of
  $(\alpha_1,\alpha)$-compatible vertex colorings may be empty: for
  $\sigma=(1,2)(3,4)(5,6)$, $\alpha=(2,3)(4,5)(1,6)$ and
  $\alpha_1=(2,3)(4,5)$ all vertices belong to the same connected
  component of $(\sigma,\alpha_1)$ (the underlying graph is a path of
  length $2$), hence all vertices must have the same color. On the other
  hand, the edge $(1,6)$ belongs to $\alpha$ only. Hence 
  the vertices $(1,2)$ and $(5,6)$ must also have different colors.
\end{remark}  

\begin{lemma}
\label{lem:alphau}
Let $(\sigma,\alpha)$ be a collection of hypermaps such that each cycle
of $\alpha$ has length at most $3$. Then for every coloring of the
cycles of $\sigma$ there is a unique refinement $\alpha_1$ of $\alpha$
such that the coloring is $(\alpha_1,\alpha)$-compatible. 
\end{lemma}  
\begin{proof}
Each cycle of $\alpha_1$ must be contained in some cycle of
$\alpha$. Consider a cycle $(i_1,\ldots,i_k)$ of $\alpha$ (where $k\leq
3$). The only way to choose the cycles of $\alpha_1$ contained in
$\alpha$ is to partition the set $\{i_1,\ldots,i_k\}$ into cycles in
such a way that each cycle is formed exactly by all points belonging to
vertices of the same color. Such a partitioning is possible, because
every partition of a set of size at most $3$ is a noncrossing partition. 
\end{proof}

Combining Equation~(\ref{eq:chromaticg}) with Lemma~\ref{lem:alphau} we
can show the following.
\begin{proposition}
\label{prop:3chr}  
Let $(\sigma,\alpha)$ be a collection of hypermaps such that each cycle
of $\alpha$ has length at most $3$. Then for each refinement $\alpha_1$ of
$\alpha$ and each nonnegative integer $m$, $X([\alpha_1,\alpha];m)$ is the
number of $(\alpha_1,\alpha)$-compatible colorings of the cycles of
$\sigma$ with $m$ colors.
\end{proposition}
\begin{proof}
We proceed by induction on $z(\alpha_1)-z(\alpha)$. The basis of the
induction is the case $\alpha_1=\alpha$: the
$(\alpha,\alpha)$-compatible colorings are exactly the ones that are
constant on each connected component of $(\sigma,\alpha)$. The number of
such colorings is $m^{\kappa(\sigma,\alpha)}$. For the induction step,
observe that $m^{\kappa(\sigma,\alpha_1)}$ is the number of all
$m$-colorings of the vertex set that are constant on the connected
components of $(\sigma,\alpha_1)$. By Lemma~\ref{lem:alphau} each such
coloring is $(\beta,\alpha)$-compatible for a unique refinement $\beta$
of $\alpha_1$ and the fact that two vertices incident to the same cycle
of $\alpha_1$ have the same color forces $\alpha_1\leq
\beta$. Conversely, if an $m$-coloring of the cycles of $\sigma$ is
$(\beta,\alpha)$-compatible for some $\beta\in ]\alpha_1,\alpha_2]$ then
    it must be constant on the connected components of
    $(\sigma,\alpha_1)$. Substituting $\alpha_2=\alpha$ and $t=m$ into
    Equation~(\ref{eq:chromaticg}) yields 
    $$
X([\alpha_1,\alpha];m)
  =
n^{\kappa(\sigma,\alpha_1)}-\sum_{\alpha_1<\beta\leq\alpha}
X([\beta,\alpha];m) .
$$
By our induction hypothesis, on the right hand side we have the number
of all $m$-colorings of the vertex set that are constant on each
connected component of $(\sigma,\alpha_1)$ but which are not
$(\beta,\alpha)$-compatible for any refinement $\beta$ of $\alpha$
properly containing $\alpha_1$ as a refinement. By subtraction
$X([\alpha_1,\alpha];m)$ is then the number of
$(\alpha_1,\alpha)$-compatible $m$-colorings of the vertex set.
\end{proof}  
The most important special case of Proposition~\ref{prop:3chr} is the
following result. 
\begin{theorem}
\label{thm:3chr}  
Let $(\sigma,\alpha)$ be a collection of hypermaps such that each cycle
of $\alpha$ has length at most $3$. Then for any positive integer $m$,
the number $m^{\kappa(\sigma,\alpha)}\cdot \chi(\sigma,\alpha,m)$ is the
number of ways to $m$-color the vertices of $(\sigma,\alpha)$ in such a way that
no two vertices of the same color are incident to the same cycle of $\alpha$.
\end{theorem}
\begin{proof}
Setting $\alpha_1$ equal to the identity permutation in
Proposition~\ref{prop:3chr} implies that $X([\id,\alpha],m)$ is the
number proper $m$-colorings we want to count. By~\eqref{eq:charchr} this
number also equals $m^{\kappa(\sigma,\alpha)}\cdot
\chi(\sigma,\alpha,m)$.      
\end{proof}  

The above reasoning cannot be extended to arbitrary hypermaps, as we
can see it in the following example.
\begin{example}
\label{ex:4cycle}  
Consider the hypermap $(\sigma,\alpha)$ where $\sigma=(1)(2)(3)(4)$ and
$\alpha=(1,2,3,4)$. Set $m=2$, color $(1)$ and $(3)$ with the same
first color, and color $(2)$ and $(4)$ with the same second color. The
only permutation of $\{1,2,3,4\}$ whose cycles connect the vertices of
the same color would be $\beta=(1,3)(2,4)$ but this permutation is not a
refinement of $\alpha$: the cycle $(1,3)$ crosses the cycle $(2,4)$ in
the cyclic order of $\alpha$.
\end{example}

\section{The flow polynomial of a hypermap}
\label{sec:flowpoly}

In this section with generalize the flow polynomial to hypermaps. 

\begin{definition}
We define the {\em flow polynomial $C(\sigma,\alpha;t)$} of a
collection of hypermaps
$(\sigma,\alpha)$ on the set of points $\{1,2,\ldots,n\}$ by
$$
C(\sigma,\alpha;t)=\sum_{\beta\leq \alpha} \mu(\beta,\alpha)
t^{n+\kappa(\sigma,\beta)-z(\beta)-z(\sigma)}. 
$$
\end{definition}  
If $(\sigma,\alpha)$ is a collection of maps, as noted in
Remark~\ref{rem:maps}, we get
$$
C(\sigma,\alpha;t)=\sum_{\beta\leq \alpha} (-1)^{z(\beta)-z(\alpha)}
t^{n+\kappa(\sigma,\beta)-z(\beta)-z(\sigma)}, 
$$
which can easily be seen to be equivalent to the usual definition of the
flow polynomial. For any collection of hypermaps, the above definition,
combined with the M\"obius inversion formula yields
\begin{align*}
  \sum_{\beta\leq \alpha} C(\sigma,\beta;t) &=
  \sum_{\beta\leq \alpha} \sum_{\beta'\leq \beta} \mu(\beta',\beta)
  t^{n+\kappa(\sigma,\beta')-z(\beta')-z(\sigma')}\\
  &=\sum_{\beta'\leq \beta}
  t^{n+\kappa(\sigma,\beta')-z(\beta')-z(\sigma')} \sum_{\beta\in
    [\beta',\alpha]}  \mu(\beta',\beta)\\
  &=\sum_{\beta'\leq \beta}
  t^{n+\kappa(\sigma,\beta')-z(\beta')-z(\sigma')} \delta_{\beta',\alpha},
\end{align*}  
that is,
\begin{equation}
\label{eq:flow}
\sum_{\beta\leq \alpha}
C(\sigma,\beta;t)=t^{n+\kappa(\sigma,\alpha)-z(\alpha)-z(\sigma)}. 
\end{equation}

For collections of hypermaps whose hyperedges have at most $3$ points,
equation~(\ref{eq:flow}) may be interpreted as a result about actual
flows, as defined by Cori and Mach\`\i~\cite{Cori-Machi-flow}.
\begin{definition}
Given a field $K$ and a collection of hypermaps $(\sigma,\alpha)$ on the
set of points $\{1,2,\ldots,n\}$, a {\em flow} on $(\sigma,\alpha)$ is a
function $f:\{1,2,\ldots,n\}\rightarrow K$ such that the set $C$ of
points of any cycle of $\sigma$ or $\alpha$ satisfies
\begin{equation}
\label{eq:flowdef}
\sum_{i\in C} f(i)=0.
\end{equation}
\end{definition}  
It is easy to verify that for maps, the above definition specializes to
the usual definition of a flow (circulation). The set of all flows on a
hypermap forms a vectorspace under pointwise addition and scalar
multiplication. It has been shown by Cori and Mach\`\i~\cite[Lemma
  2]{Cori-Machi-flow} that the dimension of the space of flows is
$n+1-z(\sigma)-z(\alpha)$. Their reasoning is applicable to collections
of hypermaps componentwise, and we obtain the following statement.
\begin{lemma}[Cori--Mach\`\i]
For a collection of hypermaps $(\sigma,\alpha)$ on the set
$\{1,2,\ldots,n\}$ the dimension of the space of flows is
$n+\kappa(\sigma,\alpha)-z(\sigma)-z(\alpha)$.
\end{lemma}  
\begin{corollary}
\label{cor:nflows}
If $K$ is a finite field with $q$ elements and $(\sigma,\alpha)$  is a
collection of hypermaps  on the set $\{1,2,\ldots,n\}$, then the number
of flows $f:\{1,2,\ldots,n\}\rightarrow K$ is
$q^{n+\kappa(\sigma,\alpha)-z(\sigma)-z(\alpha)}$.   
\end{corollary}  

Using equation~(\ref{eq:flowdef}) we may generalize Tutte's
results~\cite{Tutte} on the number of nowhere zero flows as follows.
\begin{definition}
Let $K$ be a field and $(\sigma,\alpha)$ be a collection of hypermaps on
the set $\{1,2,\ldots,n\}$. We say that the flow
$f:\{1,2,\ldots,n\}\rightarrow K$ is
a {\em nowhere zero flow} if $f(i)=0$ implies that $(i)$ is a cycle of
$\alpha$. 
\end{definition}  
\begin{remark}
Note that in the case when $(i)$ is a cycle of $\sigma$ or $\alpha$,
equation~\eqref{eq:flowdef} forces $f(i)=0$. Consider the special case when
$(\sigma,\alpha)$ is a map. In the case when $(i)$ is a
cycle of $\sigma$, the underlying graph
has a leaf vertex and there is no nowhere zero flow on the graph. On
the other hand, deleting and contracting edges creates fixed points of
$\alpha$ in our setting, hence we do not want to say a flow is nowhere
zero just because it is zero on a fixed point of $\alpha$. 
\end{remark}

\begin{theorem}
\label{thm:nzflow}  
Let $K$ be a finite field with $q$ elements and let $(\sigma,\alpha)$ be
a collection of hypermaps on the set $\{1,2,\ldots,n\}$ such that each
cycle of $\alpha$ has at most three elements. Then $C(\sigma,\alpha;q)$
is the number of nowhere zero flows $f:\{1,2,\ldots,n\}\rightarrow K$
on $(\sigma,\alpha)$. 
\end{theorem}  
\begin{proof}
By Corollary~\ref{cor:nflows} the number of all flows on
$(\sigma,\alpha)$ is
$q^{n+\kappa(\sigma,\alpha)-z(\sigma)-z(\alpha)}$. The statement may be
shown by induction on $n-z(\alpha)$ using \eqref{eq:flow} once we prove
the following statement: for each flow $f:\{1,2,\ldots,n\}\rightarrow K$
there is a unique refinement $\beta$ of $\alpha$ such that $f$ is a
nowhere zero flow on $(\sigma,\beta)$. To show this statement note that
whenever $f(i)=0$ holds, $f$ is also a flow on $(\sigma,
\alpha(\alpha^{-1}(i),i))$ obtained by deleting $(\alpha^{-1}i,i)$. This
amounts to removing $i$ from the cycle of $\alpha$ containing it and
adding $i$ as a bud. Performing this operation on an edge $(i,j)$
results in a pair of buds $(i)(j)$, performing it on a hyperedge
$(i,j,k)$ results in a pair of an edge and of a bud $(i)(j,k)$. We may
repeat this refinement operation for each zero of $f$ in any order and
we always arrive at the same refinement $\beta$ that has the following
property: $f(i)=0$ holds if and only if $i$ is a fixed point of
$\beta$. Note that $\beta$ can not be refined further: if
$(i,\ldots,j)$ is a cycle of $\beta$ of length $2$ or $3$ and
$f(i)\cdots f(j)\neq 0$ holds, then any refinement of this cycle
would contain at least one fixed point, without loss of generality
$i$, and $f$ can not be a flow on this refinement because of $f(i)\neq 0$. 
\end{proof}
The key property we used in the proof of Theorem~\ref{thm:nzflow} is the
following: any refinement of a cycle of length at most $3$ contains a
fixed point. For larger hyperedges we may encounter difficulties as
exhibited in the example below.
\begin{example}
  Consider the hypermap $(\sigma,\alpha)$ on the set 
  $\{1,2,\ldots,8\}$, given by 
  $\sigma=(1,5)(2,6)(3,7)(4,8)$ and
  $\alpha=(1,2,3,4)(5,6)(7,8)$. It is easy to verify that the
  function $f:\{1,2,\ldots,8\}\rightarrow K$ given by $f(i)=(-1)^i$ is a
  nowhere zero flow on $(\sigma,\alpha)$. However $f$ is also a nowhere
  zero flow for $(\sigma,\beta_1)$ and $(\sigma,\beta_2)$ where
  $\beta_1$ is the refinement $\beta_1=(1,2)(3,4)(5,6)(7,8)$
  and $\beta_2$ is the refinement $\beta_2=(1,4)(2,3)(5,6)(7,8)$.
\end{example}

For a planar hypermap, equation~(\ref{eq:flow}) may be
rewritten as    
\begin{equation}
\label{eq:flow-planar}
\sum_{\beta\leq \alpha}
tC(\sigma,\beta;t)=t^{z(\alpha^{-1}\sigma)}. 
\end{equation}  
For collections of maps, (\ref{eq:flow-planar}) has the following
combinatorial interpretation. For a positive integer $t$, the right hand
side counts the number of ways of coloring the faces of
$(\sigma,\alpha)$. On the left hand side, $\beta$ represents the set of
edges that have faces of the same color on both sides. This reasoning
can be extended to hypermaps with hyperedges containing cycles of length
at most three. The details are left to the reader, as the reasoning is
not only analogous to that on the chromatic polynomial, but also dual to
it. The well-known duality of the chromatic an flow polynomials of
planar graphs may be extended to planar hypermaps with hyperedges
containing at most three points.

\section*{Acknowledgments}
The second author wishes to express his heartfelt thanks to Labri, Universit\'e
Bordeaux I, for hosting him as a visiting researcher in Summer 2019 and
in Summer 2023, where a great part of this research was performed. This
work was partially supported by a grant from the Simons Foundation 
(\#514648 to G\'abor Hetyei).

\end{document}